\documentclass[12pt]{amsart}
\usepackage{url} % to display urls correctly
\usepackage{hyperref}
\usepackage{amsfonts,amsthm,latexsym,amsmath,amssymb,amscd,amsmath,epsf} %for math functions and
\usepackage{tikz}

\theoremstyle{plain}
\newtheorem{theorem}{\bf Theorem}[section]
\newtheorem{lemma}[theorem]{\bf Lemma}
\newtheorem{proposition}[theorem]{\bf Proposition}
\newtheorem{corollary}[theorem]{\bf Corollary}

\theoremstyle{definition}

\theoremstyle{remark}
\newtheorem{remark}[theorem]{Remark}

\numberwithin{equation}{section}

\makeindex

\newcommand{\kk}{\mathbf{k}}
\newcommand{\xx}{\mathbf{x}}
\newcommand{\yy}{\mathbf{y}}

\newcommand{\QQ}{\mathbb{Q}}
\newcommand{\ZZ}{\mathbb{Z}}
\newcommand{\PP}{\mathbb{P}}
\newcommand{\NN}{\mathbb{N}}
\newcommand{\CC}{\mathbb{C}}
\newcommand{\B}{\mathbb{B}}
\newcommand{\WC}{\mathbb{WCOMP}}
\newcommand{\MM}{\mathcal{M}}
\newcommand{\E}{\mathcal{E}}
\newcommand{\A}{\mathcal{A}}
\newcommand{\W}{\mathcal{W}}
\newcommand{\LL}{\mathcal{L}}

\renewcommand{\S}{\mathcal{S}}
\renewcommand{\L}{\mathbf{\varLambda}}

\newcommand{\sym}{\mathfrak{S}}
\newcommand{\ninc}{\textsf{Ninc}}
\newcommand{\clr}{\mathbf{color}}

\def\newop#1{\expandafter\def\csname #1\endcsname{\mathop{\rm #1}\nolimits}}

\newop{Hom}
\newop{supp}
\newop{col}
\newop{sgn}
\newop{des}
\newop{asc}
\newop{wcomp}
\newop{comp}
\newop{top}
\newop{ch}
\newop{c}

\DeclareMathSymbol{\subfactorial}{\mathord}{operators}{"3C}
\title{The colored symmetric and exterior algebras}
\author[R. S. Gonz\'alez D'Le\'on]{Rafael S. Gonz\'alez D'Le\'on}
\address{Escuela de Ciencias Exactas e Ingenier\'ia, Universidad Sergio Arboleda, 
Bogot\'a, Colombia}
\email{rafael.gonzalezl@usa.edu.co}
\urladdr{\url{http://dleon.combinatoria.co}}

\begin{document}

\begin{abstract}
We study colored generalizations of the symmetric algebra and its Koszul dual, the exterior 
algebra. The symmetric group $\mathfrak{S}_n$ acts on the multilinear components of these algebras. 
While $\mathfrak{S}_n$ acts trivially on the multilinear components of the colored 
symmetric algebra, we use poset topology techniques to understand the representation on its Koszul 
dual. We introduce an $\mathfrak{S}_n$-poset of weighted subsets that we call the weighted boolean 
algebra and we prove that the multilinear components of the colored exterior algebra are 
$\mathfrak{S}_n$-isomorphic to the top cohomology modules of its maximal intervals.
We use a technique of Sundaram to compute group representations on Cohen-Macaulay 
posets to give a generating formula for the Frobenius series of the colored exterior 
algebra. We exploit that formula to find an explicit expression for the expansion of the 
corresponding representations in terms of irreducible $\mathfrak{S}_n$-representations. We show that 
the two colored Koszul dual algebras are Koszul in the sense of Priddy.

\end{abstract}

\maketitle

\section{Introduction}\label{section:introduction}
Let $\kk$ denote an arbitrary field of characteristic not equal $2$ and $V$ be a finite 
dimensional $\kk$-vector space. The 
\emph{tensor algebra} 
$T(V)=\bigoplus_{n\ge 0} V^{\otimes n}$ is the free associative algebra generated by $V$, where 
$V^{\otimes n}$ denotes the tensor product of $n$ copies of $V$ and where $V^{\otimes 0}:=\kk$. For 
any set $R \subseteq T(V)$ 
denote by $\langle R \rangle$ the ideal of $T(V)$ generated by $R$.
Let $R_1$ be the subspace of $V\otimes V$ generated by the set of relations of 
the form
\begin{align}
 x\otimes y - y \otimes x\quad\quad\text{(symmetry),}\label{relation:sym1}
\end{align} for all $x,y\in V$.
The \emph{symmetric algebra} $\S(V)$ is the quotient algebra $$\S(V):=T(V)/\langle R_1\rangle.$$ 
Now let $R_2 \subseteq V\otimes V$ be generated by the set of relations of the form
\begin{align}
 x\otimes y + y\otimes x\quad\quad\text{(antisymmetry),}\label{relation:antisym1}
\end{align} for all $x,y\in V$.
The \emph{exterior algebra} $\L(V)$ is the algebra $$\L(V):=T(V)/\langle R_2\rangle.$$

We will use the concatenation $xy$ to denote the image of $x \otimes y$ in $\S(V)$ and the 
\emph{wedge} $x\wedge y$ to denote the image of $x\otimes y$ in $\L(V)$ under the canonical 
epimorphisms.

Let $V^*:=\Hom(V,\kk)$ denote the vector space dual to $V$. For finite dimensional $V$ we 
have that $V^*\simeq V$. Recall that for an associative algebra 
$A=A(V,R):=T(V)/\langle R \rangle$ generated on a finite dimensional vector space $V$ and 
(quadratic) relations 
$R \subseteq V^{\otimes 2}$
there is another algebra $A^!$ associated to $A$ that is called the \emph{Koszul dual associative 
algebra} to $A$.
Indeed, when $V$ is finite dimensional, there is a canonical isomorphism $(V^{\otimes 
2})^*\simeq V^*\otimes V^*$ and we let $R^{\perp}$ be the image under this isomorphism of the space 
of elements in $(V^{\otimes 2})^*$ that vanish on $R$. The \emph{Koszul dual} $A^!$ of $A$ is the 
algebra $A^!:=A(V^*,R^{\perp})=T(V^*)/\langle R^{\perp} \rangle$. It is known and easy to check 
from the relations (\ref{relation:sym1}) and (\ref{relation:antisym1}) (see for example 
\cite{Froberg1999}) that 
$\L(V^*)$ is the Koszul dual associative algebra to $\S(V)$.

Denote $[n]:=\{1,2,\dots,n\}$ and let $V=\kk\{[n]\}$ be the vector space with generators $[n]$. We 
define the \emph{multilinear component} $\S(n)$ as the subspace of $\S(V)$ linearly generated by 
products 
of the form $\sigma(1)\sigma(2)\cdots\sigma(n)$ where $\sigma$ is a 
permutation in the symmetric group $\sym_n$. Similarly $\L(n)$ is defined to be the subspace of 
$\L(V)$ linearly generated by \emph{wedged permutations}, i.e., the 
generators are of the form $\sigma(1)\wedge\sigma(2)\wedge\cdots\wedge\sigma(n)$ for 
$\sigma\in\sym_n$. The symmetric group acts on the generators of $\S(n)$ and $\L(n)$ by permuting 
their letters and this action induces representations of $\sym_n$ in both $\S(n)$ and $\L(n)$. 
Using 
the relations 
(\ref{relation:sym1}) and (\ref{relation:antisym1}) we can see that both $\S(n)$ and 
$\L(n)$ are always one-dimensional spaces with bases given by $\{12\cdots n\}$ and $\{1\wedge 2 
\wedge \cdots \wedge n\}$ respectively. Moreover,
for $n\ge 1$ it is easy to see that
 \begin{align*}\S(n)\cong_{\sym_n}\mathbf{1}_n \text{ and } 
\L(n)\cong_{\sym_n}\sgn_n,\end{align*}
 where $\mathbf{1}_n$ and $\sgn_n$ are respectively the trivial and the sign representations of 
$\sym_n$.

\subsection{Colored symmetric and exterior algebras}
Let $\NN$ denote the set of nonnegative integers and $\PP$ the set of positive integers. For a 
subset $S\subseteq \PP$ we 
consider the set $[n]\times S$ of colored letters of the form $(x,i)$ (that we will 
denote $x^i$ across this article) where $x \in [n]$  and $i \in S$. 
Let $V=\kk\{[n]\}$ and $V^S=\kk \{[n]\times S\}$ where $S\subseteq \PP$ is 
finite  and let $CR_1\subseteq V^S\otimes V^S$ be generated by
\begin{align}
 x^i\otimes y^j - y^i \otimes x^j\quad\quad\text{(label symmetry),}\label{relation:colsym1}\\
 x^i\otimes y^j - x^j \otimes y^i\quad\quad\text{(color symmetry),}\label{relation:colsym2}
\end{align} for all $x,y\in [n]$ and $i,j \in S$.
The \emph{$S$-colored symmetric algebra} $\S_S(V)$ is defined to be the algebra 
$$\S_S(V):=T(V)/\langle CR_1 \rangle.$$ We define the 
\emph{$S$-colored exterior algebra} $\L_S(V^*)$ on $V^*$ as the Koszul dual to 
$\S_S(V)$. Explicitly, the reader can check that if we let $CR_2\subseteq V^S\otimes V^S$ be 
generated by
\begin{align}
 x^i\otimes y^i + y^i \otimes 
x^i\quad\quad\text{(monochromatic antisymmetry),}\label{relation:colantisym1}\\
 x^i\otimes y^j + y^i \otimes x^j + y^j \otimes x^i + x^j \otimes y^i  \quad\quad\text{(mixed 
antisymmetry),}\label{relation:colantisym2}
\end{align} for all $x,y\in [n]$ and $i,j \in S$, then  $$\L_S(V)=T(V)/\langle CR_2 \rangle.$$
Choosing $S=[k]$ for some $k \in \PP$ and letting $k$ grow to infinity we obtain the  
\emph{colored 
symmetric algebra} $\S_{\PP}(V)$ and the  \emph{colored exterior algebra} $\L_{\PP}(V)$ as inverse 
limits of the above constructions.
We denote $\S_{\PP}(n)$ and $\L_{\PP}(n)$ respectively the components of $\S_{\PP}(V)$ and 
$\L_{\PP}(V)$ linearly generated by colored permutations and wedged colored permutations. A 
\emph{colored permutation} is a permutation $\sigma \in \sym_n$ together with a function 
that assigns to each $x\in [n]$ a \emph{color} $\clr(x)\in \PP$. For example 
${\color{blue}2^1}{\color{orange}1^4}{\color{red}3^2}$ is a colored permutation of $[3]$ (here 
$\clr(1)=4$, $\clr(2)=1$ and $\clr(3)=2$) and so a 
generator in $\S_{\PP}(3)$. For a colored permutation $\sigma$ of $n$ let $\wedge(\sigma)$ 
denote the \emph{wedged colored permutation} $\sigma(1)\wedge\sigma(2)\wedge\cdots\wedge\sigma(n)$. 
For example 
$\wedge({\color{blue}2^1}{\color{orange}1^4}{\color{red}3^2})={\color{blue}2^1}\wedge{\color{orange}
1^4}\wedge{\color{red}3^2}$ is a generator in  
$\L_{\PP}(3)$. Let $\sym_n^S$ denote the set of colored permutations of $[n]$ with colors in $S 
\subseteq \PP$.

A \emph{weak composition} $\mu$ of $n$  is a sequence of nonnegative
integers $(\mu(1),\mu(2),\dots)$ such that $|\mu|:=\sum_{i\ge 1}\mu(i)=n$. 
Let $\wcomp$ be the set of weak compositions and $\wcomp_n$ the set of weak compositions of $n$. 
For example, $(0,1,2,0,1):=(0,1,2,0,1,0,0,\dots)$ is in $\wcomp_4$. In this work weak compositions 
will be the combinatorial device used to track the number of occurrences of each color among the 
letters of the generators.
Indeed, for $\sigma \in \sym_n^{\PP}$, let $\mu(\sigma) \in \wcomp$ be such that $\mu(\sigma)(j)$ 
is the number of letters of color $j$ in $\sigma$ for each $j$, we call $\mu(\sigma)$ the  
\emph{content} 
of $\sigma$. For example 
$\mu({\color{cyan}2^3}{\color{green}1^5}{\color{red}3^2}{\color{cyan}4^3})=(0,1,2,0,1)$. For $\mu 
\in \wcomp$ we denote by $\sym_{\mu}$ the set of colored permutations of 
content $\mu$. Define $\S(\mu)$ and $\L(\mu)$ to be respectively the subspaces 
of $\S_{\PP}(|\mu|)$ and $\L_{\PP}(|\mu|)$ generated by colored permutations and wedged 
colored permutations in $\sym_{\mu}$. For example $\S(0,1,2,0,1)$ and $\L(0,1,2,0,1)$ 
have generators associated with colored permutations of $[4]$ that contain one letter of color $2$, 
two letters of color $3$, one letter of color $5$ and no other letters. The 
symmetric group $\sym_n$ acts on $\S_{\PP}(n)$ and $\L_{\PP}(n)$ as before. A 
permutation $\tau \in \sym_n$ acts replacing the colored letter $x^i$ by the colored letter 
$\tau(x)^i$. This action preserves the colors of the generators and so $\S(\mu)$ and 
$\L(\mu)$ are also representations of $\sym_n$. Naturally if $\nu$ is a rearrangement of 
$\mu$, 
we have that $\S(\nu)\simeq_{\sym_n}\S(\mu)$ and $\L(\nu)\simeq_{\sym_n}\L(\mu)$. In 
particular, if $\mu$ has a single nonzero component then $\S(\mu)$ is isomorphic to $\S(n)$ and 
$\L(\mu)$ is isomorphic to $\L(n)$. 

For $\mu \in \wcomp_n$ define its \emph{support} $\supp(\mu)=\{j\in \PP \,\mid\,
\mu(j) \ne 0\}$. Then for $S \subseteq \PP$ we have
\begin{align*}
 \S_S(n)\simeq_{\sym_n}\bigoplus_{\substack{\mu \in \wcomp_n\\ \supp(\mu) \subseteq S}}
\S(\mu)\quad\text{ and }\quad
\L_S(n)\simeq_{\sym_n}\bigoplus_{\substack{\mu \in \wcomp_n\\ \supp(\mu) \subseteq S}}
\L(\mu).
\end{align*}

The following theorem follows immediately from relations (\ref{relation:colsym1}) and 
(\ref{relation:colsym2}).
\begin{theorem}\label{theorem:Smu_trivial}
 For $n\ge 1$ and $\mu \in \wcomp_n$,
 $$\S(\mu)\cong_{\sym_n}\mathbf{1}_{n}.$$
\end{theorem}

Our goal is to understand the more interesting representation of $\sym_n$ on $\L(\mu)$ for 
all 
$\mu \in \wcomp$. In order to accomplish this we are going to apply the program started by Hanlon 
and Wachs in \cite{HanlonWachs1995} and by Wachs in \cite{Wachs1998}. They applied poset topology 
techniques to the partially ordered set (or \emph{poset}) $\Pi_n$ of partitions  of the set $[n]$, 
and to the induced subposet of $\Pi_n$ where all partitions have parts of size congruent to $1$ mod 
$(k-1)$, in order to understand algebraic properties of the multilinear components of the free Lie 
algebra and the free Lie $k$-algebra (see also \cite{Barcelo1990}). Gottlieb and Wachs 
\cite{GottliebWachs2000} have extended 
the results on the poset of partitions to more general Dowling lattices. 
The author and Wachs 
\cite{DleonWachs2013a} and the author \cite{Dleon2014} have applied similar techniques to a family 
of posets of weighted partitions in their study of the operad of Lie algebras with multiple 
compatible brackets and its Koszul dual operad, the operad of commutative algebras with multiple 
totally commutative products (see also \cite{DotsenkoKhoroshkin2007,Liu2010}). The original 
motivation for the present work is precisely the study 
of analogous constructions within the category of connected graded associative algebras.

The main idea of the technique 
in \cite{Dleon2014,DleonWachs2013a,GottliebWachs2000,HanlonWachs1995,Wachs1998} is that in 
order to study the representation of $\sym_n$ on the multilinear component $A(n)$ of certain 
algebra 
$A$, a certain poset $P_A$ is constructed so that the (co)homology of $P_A$ (defined later) is 
$\sym_n$-isomorphic to $A(n)$ (maybe up to tensoring with the sign representation). Then poset 
topology techniques applied to $P_A$ can recover algebraic information about $A(n)$. 

\section{Main results}\label{section:mainresults}
To every poset $P$ one can associate a simplicial complex $\Delta(P)$ (called the \emph{order 
complex}) whose faces are the chains (totally ordered subsets) of $P$. If there is a group $G$ 
acting on the poset in such a way that every $g\in G$ is a (strict) order preserving map on $P$ 
then this 
action induces isomorphic representations of $G$ on the $j$-th reduced simplicial
homology $\widetilde H_{j}(P)$ and cohomology $\widetilde H^{j}(P)$  of the order complex 
$\Delta(P)$ over the field $\kk$. Let $P$ be 
\emph{bounded} (it has a unique minimal element, denoted $\hat{0}$, and a unique maximal 
element, denoted $\hat{1}$) and \emph{pure} (all the maximal chains have the 
same length). If for every open interval $(x,y)$ in $P$ it happens that $\widetilde H_i((x,y)) = 
0$ for all $x <y$ in $P$ and $i < l([x,y])-2$ we say that $P$ is \emph{Cohen-Macaulay}. Some poset 
topology techniques on pure bounded posets, like the theory of lexicographic shellability, imply 
Cohen-Macaulayness (see \cite{Wachs2007}). 

Recall that the \emph{boolean algebra} $\B_n$ is the poset of subsets of $[n]$ 
ordered by inclusion. 
It is known that 
$\B_n$ is Cohen-Macaulay (see for example \cite{Bjorner1980}) and there is a natural action of 
$\sym_n$ on 
$\B_n$ permuting the elements in $[n]$. This action induces a representation of $\sym_n$ on the 
unique nonvanishing reduced 
simplicial cohomology $\widetilde 
H^{n-2}(\overline{\B}_n)$  of 
the proper part $\overline{\B}_n:=\B_n \setminus \{\hat{0},\hat{1}\}$ of $\B_n$.
The following isomorphism is already a classical result (see  
\cite{Solomon1968}),
 \begin{align}\label{equation:isosolomon}\widetilde H^{n-2}(\overline{\B}_n)\cong_{\sym_n}
\sgn_n \quad(\cong_{\sym_n} \L(n)).\end{align}

To use the technique mentioned above we need to construct the poset $P_A$ 
associated to the algebra $A$ such that the (co)homology of $P_A$ is 
$\sym_n$-isomorphic to $A(n)$ . For certain algebras there is a recipe to cook up the 
poset $P_A$. M\'endez and 
Yang 
\cite{MendezYang1991} have developed a way to associate a family of posets to an 
injective monoid in the monoidal categories of species with respect to the operations of 
product and composition. Vallette \cite{Vallette2007} has 
rediscovered the construction in \cite{MendezYang1991} for basic set operads (injective monoids 
with respect to composition) and used it to develop a criterion for the Koszulness of the 
associated operad and its Koszul dual operad by checking that the maximal intervals in the 
associated 
posets are Cohen-Macaulay (see also Fresse \cite{Fresse2004}). An operad is an 
algebraic structure that encodes types of algebras (see \cite{LodayVallette2012}). In analogy with 
\cite{Vallette2007}, Mendez in
\cite{Mendez2010} has developed a criterion for Koszulness of an associative algebra (with the 
left cancellative property) and its 
Koszul dual associative algebra. 

We will be following very closely the thread of ideas in \cite{Dleon2014}. We will 
recall here some of the concepts involved while referring the reader to consult \cite{Dleon2014} 
for 
most of the background and related notation.

\subsection{The weighted boolean algebra}
For undefined poset notation and terminology the reader can consult \cite{Stanley2012}.
Let $\WC$ be the 
partially ordered set of weak compositions with order relation defined as follows: for every 
$\nu,\mu \in \wcomp$, we say that $\mu \le \nu$ if $\mu(i)\le \nu(i)$
for every $i$. We define $\WC_n$ to be the induced subposet of $\WC$ whose elements are weak 
compositions $\mu \in \wcomp$ with $|\mu| \le n$.
A \emph{weighted subset} of $[n]$ is a set $B^{\mu}$ where $B \subseteq [n]$ and $\mu \in 
\wcomp_{|B|}$. 
Since weak compositions are infinite vectors we can use component-wise addition 
and subtraction, for instance, we denote by $\nu + \mu$, the weak composition defined by 
$(\nu + \mu)(i):=\nu(i) + \mu(i)$.

The \emph{weighted boolean algebra} $\B_n^{w}$  is the partially ordered set (poset) of 
weighted subsets of $[n]$ with  order relation given by $A^{\mu}\le B^{\nu}$ if the 
following conditions hold:

\begin{itemize}
 \item $A \le B$ in $\B_n$ and,
 \item $\mu \le \nu$ in $\WC_n$.
\end{itemize}

Equivalently, we can define the cover relation $A^{\mu}\lessdot B^{\nu}$ by:

\begin{itemize}
 \item $A \lessdot B$ in $\B_n$ and,
 \item $\nu-\mu=\bf{e}_r$ for some $r \in 
\PP$, where
$\bf{e}_r$ is
the weak composition
 with a $1$ in the $r$-th component and $0$ in all other entries.
\end{itemize}

For $S \in \PP$ we also denote by $\B_n^S$ the induced subposet of $\B_n^w$ whose elements are 
weighted subsets $B^\mu$ with $\supp(\mu)\subseteq S$.

In Figure~\ref{figure:weightedbooleanposetn3k2} we illustrate the \emph{Hasse diagram} of the 
poset $\B_3^{[2]}$.

\begin{figure}[ht]
\begin{center} 
\begin{tikzpicture}[line join=bevel,scale=0.6]

\tikzstyle{every node}=[inner sep=0pt, scale=0.7, minimum width=4pt]

\node (nempty) at (0,0) {$\emptyset^{(0,0)}$};

\node (n1-10) at (-5,3) {$1^{(1,0)}$};
\node (n2-10) at (-3,3) {$2^{(1,0)}$};
\node (n3-10) at (-1,3) {$3^{(1,0)}$};
\node (n1-01) at (1,3) {$1^{(0,1)}$};
\node (n2-01) at (3,3) {$2^{(0,1)}$};
\node (n3-01) at (5,3) {$3^{(0,1)}$};

\node (n12-20) at (-8,6) {$12^{(2,0)}$};
\node (n13-20) at (-6,6) {$13^{(2,0)}$};
\node (n23-20) at (-4,6) {$23^{(2,0)}$};
\node (n12-11) at (-2,6) {$12^{(1,1)}$};
\node (n13-11) at (0,6) {$13^{(1,1)}$};
\node (n23-11) at (2,6) {$23^{(1,1)}$};
\node (n12-02) at (4,6) {$12^{(0,2)}$};
\node (n13-02) at (6,6) {$13^{(0,2)}$};
\node (n23-02) at (8,6) {$23^{(0,2)}$};

\node (n123-30) at (-6,9) {$123^{(3,0)}$};
\node (n123-21) at (-2,9) {$123^{(2,1)}$};
\node (n123-12) at (2,9) {$123^{(1,2)}$};
\node (n123-03) at (6,9) {$123^{(0,3)}$};

\draw (nempty)--(n1-10);
\draw (nempty)--(n2-10);
\draw (nempty)--(n3-10);
\draw (nempty)--(n1-01);
\draw (nempty)--(n2-01);
\draw (nempty)--(n3-01);

\draw (n1-10)--(n12-20);
\draw (n1-10)--(n12-11);
\draw (n1-10)--(n13-20);
\draw (n1-10)--(n13-11);

\draw (n2-10)--(n12-20);
\draw (n2-10)--(n12-11);
\draw (n2-10)--(n23-20);
\draw (n2-10)--(n23-11);

\draw (n3-10)--(n23-20);
\draw (n3-10)--(n23-11);
\draw (n3-10)--(n13-20);
\draw (n3-10)--(n13-11);

\draw (n1-01)--(n12-02);
\draw (n1-01)--(n12-11);
\draw (n1-01)--(n13-02);
\draw (n1-01)--(n13-11);

\draw (n2-01)--(n12-02);
\draw (n2-01)--(n12-11);
\draw (n2-01)--(n23-02);
\draw (n2-01)--(n23-11);

\draw (n3-01)--(n23-02);
\draw (n3-01)--(n23-11);
\draw (n3-01)--(n13-02);
\draw (n3-01)--(n13-11);

\draw (n12-20)--(n123-30);
\draw (n12-20)--(n123-21);
\draw (n13-20)--(n123-30);
\draw (n13-20)--(n123-21);
\draw (n23-20)--(n123-30);
\draw (n23-20)--(n123-21);

\draw (n12-11)--(n123-21);
\draw (n12-11)--(n123-12);
\draw (n13-11)--(n123-21);
\draw (n13-11)--(n123-12);
\draw (n23-11)--(n123-21);
\draw (n23-11)--(n123-12);

\draw (n12-02)--(n123-12);
\draw (n12-02)--(n123-03);
\draw (n13-02)--(n123-12);
\draw (n13-02)--(n123-03);
\draw (n23-02)--(n123-12);
\draw (n23-02)--(n123-03);

\end{tikzpicture}
\end{center}
\caption{$\B_3^{[2]}$}\label{figure:weightedbooleanposetn3k2}
\end{figure}

We can define $\B_n^w$ in a different way. Note that both posets $\B_n$ and $\WC_n$ are ranked and 
hence have well-defined poset maps $rk:\B_n\rightarrow C_{n+1}$ and
$rk:\WC_n\rightarrow C_{n+1}$ to the chain with $n+1$ elements $C_{n+1}$. Recall that the 
\emph{Segre or fiber
product} $\displaystyle P\underset{f,g}{\times}Q$ of
two poset maps $f:P \rightarrow R$ and $g:Q \rightarrow R$ is the induced subposet of the product
$P \times Q$ with elements $\{(p,q)\arrowvert \, f(p)=g(q)\}$. Then we have that 
$\B_n^w=\B_n\underset{rk,rk}{\times}\WC_n$. 

The poset $\B_n^w$ has a minimum element $\hat 0:=\varnothing^{\mathbf{0}}$ and maximal elements 
 $[n]^{\mu}:=\{[n]^{\mu}\}$ indexed by 
weak compositions $\mu \in \wcomp_{n}$.
Note that for every $\nu, \mu \in \wcomp_{n}$ such that $\nu$ is a rearrangement of $\mu$, the 
maximal intervals $[\hat 0, [n]^{\nu}]$  and $[\hat 0, [n]^{\mu}]$ are isomorphic to each other.
In particular, if $\mu$ has a single nonzero component, these intervals are isomorphic to 
$\B_n$, hence $\B_n^{[1]} \simeq \B_n$. 
In the case when $S=[2]$ the poset $B_n^{[2]}$ is isomorphic to a poset introduced by Shareshian 
and Wachs in \cite{ShareshianWachs2009} and it is closely related to a poset of Bj\"orner and Welker 
in
 \cite{BjornerWelker2005}.

The symmetric group $\sym_n$ acts on $\B_n^w$ in the following way: for any $B^{\mu}\in 
\B_n^w$ and $\tau \in 
\sym_n$ we have $\tau B^{\mu}:=(\tau B)^{\mu}$ where $\tau B:=\{\tau(i) \,\mid\, i \in B\}$. 
Since any $\tau \in \sym_n$ is a strict order preserving morphism, the action of $\sym_n$ on 
$\B_n^w$ induces an action on the unique nonzero reduced (co)homology $\widetilde H ^{n-2}
((\hat{0},[n]^{\mu}))$ of the open maximal interval $(\hat{0},[n]^{\mu})$ of $\B_n^w$. In Section 
\ref{section:isomorphism} we prove the 
following isomorphism (Theorem \ref{theorem:explicitisomorphism}).

\begin{theorem}\label{theorem:isomorphism}For $n\ge 0$ and $\mu \in \wcomp_{n}$,
 \begin{align*}  
 \L(\mu) \simeq_{\sym_n}  \widetilde H^{n-2}((\hat 0, [n]^{\mu})). 
\end{align*}
\end{theorem}
 
An \emph{EL-labeling} is a certain type of labeling of the edges of the \emph{Hasse 
diagram} of a poset $P$ that satisfy certain requirements (defined in 
Section~\ref{section:ellabeling}). EL-labelings were introduced by Bj\"orner \cite{Bjorner1980} and 
further developed by Bj\"orner and Wachs \cite{BjornerWachs1983}.  A poset $P$ that admits an 
EL-labeling is said to be \emph{EL-shellable} and EL-shellability has important topological 
implications on its order complex $\Delta(P)$.
Let $\widehat{\B_n^w}:=\B_n^w \cup \{\hat{1}\}$ be the poset $\B_n^w$ after a maximal element has 
been added. In Section \ref{section:homotopytype} we prove the following theorem.
\begin{theorem}\label{theorem:ellabeling}
$\widehat{\B_n^w}$ is EL-shellable and hence Cohen-Macaulay. 
\end{theorem}

An \emph{ascent} in a colored permutation $\sigma \in \sym_n^{\PP}$ is a value $i \in [n-1]$ such 
that 
\begin{itemize}
 \item there is a regular ascent at position $i$ in the underlying 
uncolored permutation 
$\tilde \sigma \in \sym_n$, i.e., $\tilde \sigma(i) <\tilde \sigma(i+1)$; and
\item $\clr(\sigma(i))\le\clr(\sigma(i+1))$.
\end{itemize}

A \emph{nonincreasing colored permutation} is a colored permutation $\sigma \in \sym_n^{\PP}$  that 
is ascent-free. 

For example ${\color{blue}2^1}{\color{orange}1^4}{\color{red}3^2}$ is a nonincreasing colored 
permutation but ${\color{blue}2^1}{\color{red}1^2}{\color{orange}3^4}$  is not since 
the pair $({\color{red}1^2},{\color{orange}3^4})$ forms an ascent. Let $\ninc_n$ be the set of 
nonincreasing colored permutations and $\ninc_{\mu}$ the ones with content $\mu$. In Section 
\ref{section:homotopytype} using the 
EL-labeling mentioned in Theorem \ref{theorem:ellabeling} and results in the theory of 
lexicographic 
shellability (\cite{Bjorner1980,BjornerWachs1983}) we obtain the following algebraic information.

\begin{theorem} \label{theorem:dimension} For $\mu \in \wcomp$ the set
$$\{\wedge(\sigma)\,\mid\, \sigma \in \ninc_{\mu}\}$$ is a basis for $\L(\mu)$. 
Consequently,
 $$\dim\L(\mu) = |\ninc_{\mu}|.$$
\end{theorem}

Consider now the generating function
\begin{align}\label{definition:euleriansymmetric}
 \sum_{\mu \in \wcomp_{n}}\dim \L(\mu)\,\xx^{\mu},
\end{align}
where $\xx^{\mu}=x_1^{\mu(1)}x_2^{\mu(2)}\cdots$. Since $\dim \L(\mu)$ is invariant under any 
rearrangement of the parts of $\mu$ we have that (\ref{definition:euleriansymmetric}) 
belongs to the ring of symmetric functions $\Lambda_{\ZZ}$ (see  \cite{Macdonald1995} and 
\cite[Chapter 7]{Stanley1999} for the definitions). 

The symmetric function (\ref{definition:euleriansymmetric}) is also $e$-nonnegative; i.e., the 
coefficients of its 
expansion  in the basis of elementary symmetric functions are all 
nonnegative.
Indeed, we associate a \emph{type} (or integer partition) to  each $\sigma \in \sym_n$ in the 
following way: Let $\pi(\sigma)$ be the finest (set) partition of the set $[n]$ satisfying
\begin{itemize}
\item whenever $\sigma(i) < \sigma(i+1)$ for some $i \in [n-1]$, $\sigma(i)$ and $\sigma(i+1)$
belong to the same block of $\pi(\sigma)$.
\end{itemize}
We define the \emph{type} $\lambda(\sigma)$ of $\sigma$ to be the (integer) 
partition whose parts are the sizes of the blocks of $\pi(\sigma)$. For example, for the 
permutation $5126473$ the associated partition is $\lambda(\sigma)=(3,2,1,1)$. We obtain the 
following description of \eqref{definition:euleriansymmetric}.

\begin{theorem} \label{theorem:dimensionstype}
For all $n$,
\begin{align*}
  \sum_{\mu \in \wcomp_{n}}\dim \L(\mu)\,\xx^{\mu}= \sum_{\sigma \in \sym_n} 
e_{\lambda(\sigma)}(\xx),
 \end{align*} 
 where $e_{\lambda}$ is the elementary symmetric function associated with the partition 
$\lambda$.
 \end{theorem}
 
 The following theorem gives another characterization of \eqref{definition:euleriansymmetric}.

\begin{theorem}\label{theorem:multiplicativeinverse}
We have
  \begin{align*}
   \sum_{n\ge0}\sum_{\mu \in \wcomp_{n}}\dim 
\L(\mu)\,\xx^{\mu}\frac{y^n}{n!} =\left [ \sum_{n\ge0}(-1)^{n} 
h_{n}(\xx)\frac{y^n}{n!} \right ] ^{-1},
  \end{align*}
  where $h_n$  is the complete homogeneous symmetric function and $(\cdot)^{-1}$ denotes the 
multiplicative inverse of a formal power series.
 \end{theorem}
 
Even though Theorem \ref{theorem:multiplicativeinverse} can be proven directly (using for example 
the recursive definition of the M\"obius invariant of the maximal intervals of $\B_n^w$, Philip 
Hall's theorem and the isomorphism of Theorem \ref{theorem:isomorphism}),
in Section \ref{section:frobeniuscharacteristic} we prove, using 
a technique of Sundaram \cite{Sundaram1994} to compute group representations on Cohen-Macaulay 
$G$-posets, an equivariant version that reduces to Theorem \ref{theorem:multiplicativeinverse} by 
specialization. Let  $\ch V$ 
denote the Frobenius characteristic in variables $\yy=(y_1,y_2,\dots)$ of
an $\sym_n$-module $V$.
Recall that the map  $\ch : \mathnormal{Rep}_{\sym} \rightarrow \Lambda_{\ZZ}$ is an isomporphism 
between the Grothendieck ring 
 of representations of symmetric groups $\mathnormal{Rep}_{\sym}$ and the ring of symmetric 
functions $\Lambda_{\ZZ}$ where the Schur function $s_{\lambda}$ is the image under $\ch$ of the 
Specht module (irreducible $\sym_n$-module) $S^{\lambda}$ for every partition $\lambda$.

\begin{theorem}\label{theorem:representationmultiplicativeinverse}
 We have that 
\begin{align*}
 \sum_{n\ge 0} \sum_{\mu \in \wcomp_{n}}\ch \L(\mu)\,\xx^{\mu}=\Bigl 
(\sum_{n\ge
0}(-1)^n h_{n}(\xx)h_{n}(\yy)
\Bigr)^{-1}.
\end{align*}

\end{theorem}
Note that the power series in Theorem \ref{theorem:representationmultiplicativeinverse} belongs to  
the ring 
of symmetric power series in $\yy$ with coefficients in the ring $\Lambda_{\QQ}$ of symmetric 
functions in $\xx$.

We use a theorem discovered by Fr\"oberg \cite{Froberg1975}; Carlitz, Scoville and Vaughan 
\cite{CarlitzScovilleVaughan1976}; and Gessel 
\cite{Gessel1977} that provides an explicit description of the multiplicative inverse in Theorem 
\ref{theorem:representationmultiplicativeinverse} to give the following explicit formula for the 
representation of $\sym_n$ on $\L(\mu)$.

\begin{theorem}\label{theorem:representation}We have that
 $$\sum_{\mu \in \wcomp_n}\ch \L(\mu)\xx^{\mu}=\sum_{\eta \vdash n} \sum_{\alpha \in 
\comp_n}c_{H(\alpha),\eta}e_{\lambda(\alpha)}(\xx)s_{\eta}(\yy)$$
where $c_{H(\alpha),\eta}$ is the number of Young tableaux of shape $\eta$ and descent set 
$\des H(\alpha)$ (defined in Section \ref{section:frobeniuscharacteristic}) and $\comp_n$ is the set 
of integer compositions of $n$.
\end{theorem}

Theorem \ref{theorem:representation} can be seen as an equivariant version of the nonnegativity in 
the $e$ basis in Theorem 
\ref{theorem:dimensionstype}. Indeed, if we write 
$$\sum_{\mu \in \wcomp_n}\ch \L(\mu)\xx^{\mu}=\sum_{\lambda \vdash n} 
C_{\lambda}(\yy)e_{\lambda}(\xx),$$
then Theorem \ref{theorem:representation} implies that the coefficients $C_{\lambda}(\yy)$ are 
Schur-nonnegative.

In Section \ref{section:Koszul} we describe (Theorem \ref{theorem:koszulity}) how 
the technique developed by M\'endez in \cite{Mendez2010} implies that for a finite $S$ and finite 
dimensional $V$ the Koszul dual 
algebras $\S_S(V)$ and $\L_S(V)$ are Koszul in the sense of Priddy \cite{Priddy1970}. 

Finally, in Section \ref{section:specializations} we show how certain specializations of the
identity in Theorem \ref{theorem:representationmultiplicativeinverse}, and the identity obtained 
from this together with Theorem \ref{theorem:representation}, reduce to a classical formula 
for the exponential generating function of the Eulerian polynomials and also to an identity that 
involves the characters of the regular representations of $\sym_n$ for all $n$. The latter one 
revealing that the representation $ \L_S(n)$ is an extension of the regular representation 
$\CC[\sym_n]$.

\section{The isomorphism $\L(\mu) \simeq_{\sym_n}  \widetilde H^{n-2}((\hat 0, 
[n]^{\mu}))$}\label{section:isomorphism}

\subsection{A generating set for $\widetilde H^{n-2}((\hat 0, [n]^{\mu}))$}\label{section:genhom}

The top dimensional cohomology of a pure poset $P$, say of length $\ell$, has a particularly simple
description.  Let $\MM(P)$ denote the set of maximal chains of $P$ and let $\MM^\prime(P)$ denote
the set of chains of length $\ell-1$.  We view the coboundary map $\delta$ as a map from the chain
space of $P$ to itself, which takes chains of length $r$ to chains of length $r+1$ for all $r$.
It can be shown (see for example the appendix in \cite{DleonWachs2013a}) that $\delta$ acts on 
$r$-chains as follows:
\begin{align*}
\delta_r(\alpha_0<\cdots<\alpha_r)=&
\\\sum_{i=0}^{r+1}(-1)^{i}&\sum_{\alpha \in
(\alpha_{i-1},\alpha_i)}(\alpha_0<\cdots <\alpha_{i-1}<\alpha<\alpha_i<\cdots<\alpha_r),
\end{align*}
where  $\alpha_{-1}=\hat{0}$ and $\alpha_{r+1}=\hat{1}$.
Since the image of $\delta$ on the top chain space (i.e. the space spanned by $\MM(P)$) is $0$, the
kernel is the entire top chain space. Hence top cohomology is the quotient of the space spanned by
$\MM(P)$ by the image of the space spanned by $\MM^\prime(P)$.  The image of $\MM^\prime(P)$ is 
what we call the coboundary relations.
We thus have the following presentation of the top cohomology  $$\widetilde H^{\ell}(P) = \langle 
\MM(P)
| \mbox{ coboundary relations} \rangle.$$
Note that for any $\alpha=\alpha_0<\cdots<\alpha_{\ell-1} 
\in \MM^\prime(P)$ there is exactly one step $\alpha_{i-1}<\alpha_{i}$ for some 
$i=0,\dots,\ell$ where the chain $\alpha$ can be \emph{refined} (or augmented) to get a chain in 
$\MM(P)$. Then the cohomology relations in the top cohomology are generated by relations of the form
\begin{align}\label{equation:cohomologyrelations}
\sum_{\alpha \in
(\alpha_{i-1},\alpha_i)}(\alpha_0<\cdots <\alpha_{i-1}<\alpha<\alpha_i<\cdots<\alpha_d)=0.
\end{align}
\subsubsection{Description of the maximal chains}
We show that the maximal chains in a maximal interval $[\hat{0},[n]^{\mu}]$ of $\B_n^w$ 
are in bijection with colored permutations, hence we can use permutations in $\sym_{\mu}$ to 
describe the elements of  $\MM([\hat{0},[n]^{\mu}])$. A map $\bar \lambda: 
\mathcal E(P) \to \Lambda$,  where $\mathcal E(P)$ is the set of edges (covering relations) of the 
Hasse diagram of a poset $P$ and $\Lambda$ is a fixed poset is called an \emph{edge labeling}.
Note that a covering relation in $\B_n^w$ is of the form $A^{\nu}\lessdot (A\cup 
\{x\})^{\nu+\mathbf{e_i}}$ where $A^{\nu}$ is a weighted subset of $[n]$, $x\in [n]\setminus A$ and 
$i 
\in \PP$. So we can associate a labeling $\bar \lambda:\E(\B_n^w)\rightarrow [n]\times\PP$ given by
\begin{align}\label{definition:labeling}
\bar\lambda(A^{\nu},(A\cup \{x\})^{\nu+\mathbf{e_i}})=x^i, 
\end{align}
where $[n]\times\PP$ is the product poset of the totally ordered sets $[n]$ and $\PP$.
In Section~\ref{section:homotopytype} we conclude that this labeling $\bar\lambda$ is actually an 
EL-labeling of 
$\B_n^w$. Furthermore, this labeling can be extended to an EL-labeling of  
$\widehat{\B_n^w}:=\B_n^w \cup \{\hat{1}\}$ ($\B_n^w$ with a maximal element added). We 
denote by

\begin{align*}
\bar \lambda(\c) = \bar \lambda(x_0, x_1) \bar \lambda(x_1, x_2) \cdots \bar \lambda(x_{\ell-1}, 
x_{\ell}),
\end{align*} 
the word of labels corresponding to a maximal chain $\c = (\hat 0 = x_0
\lessdot x_1 \lessdot \cdots \lessdot x_{\ell-1} \lessdot x_{\ell}= \hat 1)$.
In the case $\c \in \MM([\hat{0},[n]^{\mu}])$ it is immediate that $\bar \lambda 
(\c) \in \sym_{\mu}$ since in $\c$ each letter of $[n]$ appears exactly once and the 
possible sequences of colors in $\c$ are determined by $\mu$.
For example the chain
$$\hat{0} \lessdot \{2\}^{(1,0,0,0)} \lessdot \{1,2\}^{(1,0,0,1)}\lessdot 
\{1,2,3\}^{(1,1,0,1)}$$ corresponds to  the word of labels 
${\color{blue}2^1}{\color{orange}1^4}{\color{red}3^2}$.
Clearly, starting with $\sigma \in \sym_{\mu}$ we can also recover the chain $\c \in 
\MM([\hat{0},[n]^{\mu}])$ such that $\bar \lambda 
(\c)=\sigma$.
Indeed, for $\sigma \in \sym_{\mu}$ define the chain $\c(\sigma) \in \MM([\hat{0},[n]^{\mu}])$  
to be the one whose rank $0$ element is $\hat{0}$ and whose rank $i$ weighted subset is 
$$\{\sigma(1),\sigma(2),\dots,\sigma(i)\}^{\mathbf{e_{\clr(\sigma(1))}}+\mathbf{e_{
\clr(\sigma(2))}}+\cdots+\mathbf{e_{\clr(\sigma(i))}}}$$ for all $i\in[n]$. 
We conclude the following theorem.

\begin{proposition}\label{proposition:bijectionchainscoloredpermutations}
 The maps $\bar \lambda$ and $\c$ above define a bijection $$\MM([\hat{0},[n]^{\mu}])\simeq 
\sym_{\mu}.$$
\end{proposition}
Note also that for a bounded poset $P$ the sets $\MM(P)$ and $\MM(P\setminus \{\hat{0},\hat{1}\})$ 
are in bijection by associating a chain $\c \in \MM(P)$ with the chain $\bar\c:=\c\setminus 
\{\hat{0},\hat{1}\} \in \MM(P\setminus \{\hat{0},\hat{1}\})$. For $\sigma \in \sym_{\mu}$ we write 
$\bar \c(\sigma):=c(\sigma) \setminus \{\hat 0, \hat{1}\}$ for the corresponding chain in 
$(\hat{0},[n]^{\mu})$. 

The codimension one chains in $(\hat{0},[n]^{\mu})$ are unrefinable except between a pair of 
adjacent elements in $[\hat{0},[n]^{\mu}]$ so the generating relations of equation 
(\ref{equation:cohomologyrelations}) correspond to the two different type of intervals of length $2$ 
in  $[\hat{0},[n]^{\mu}]$. The intervals of length 
$2$ happen when  two elements $x$ and $y$ have been added to a weighted 
subset $A^{a}$ and the weight has been increased accordingly. The \emph{Type I} intervals occur 
when the weight has been increased by $2\color{red}\mathbf{e_i}$ and the \emph{Type II} intervals 
when the weight has been increased by 
$\color{red}\mathbf{e_i}+\color{blue}\mathbf{e_j}$ with $i\neq j \in \PP$ (see 
Figure~\ref{figure:cohomologyrelations}).
In the following we denote $\alpha x^i y^j \beta$ a colored permutation where $\alpha$ and 
$\beta$ are the starting and trailing colored subwords.

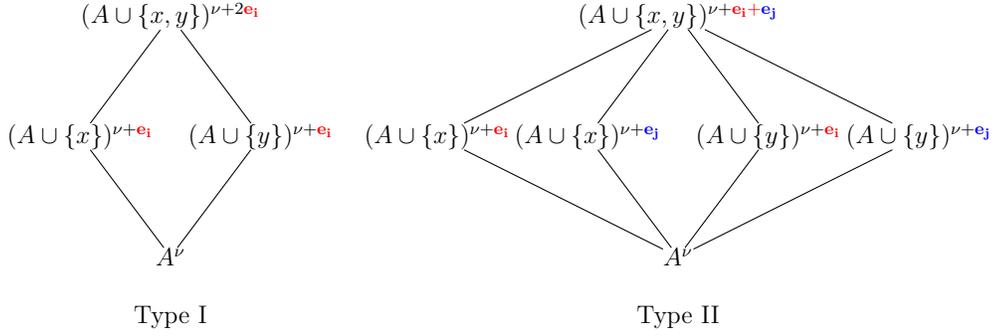
\begin{figure}[ht]
\begin{center} 
\begin{tikzpicture}[line join=bevel,scale=0.8]

\begin{scope}
  \tikzstyle{every node}=[inner sep=0pt, scale=0.80, minimum width=4pt]
  \node (v0) at (0,0)  {$A^{\nu}$};
  \node (v12) at (0,4)  {$(A\cup\{ x,y\})^ {\nu+2\color{red}\mathbf{e_i}}$};
  \node (v1) at (-1.5,2)  {$(A\cup\{ x\})^ {\nu+\color{red}\mathbf{e_i}}$};
  \node (v2) at (1.5,2)  {$(A\cup\{ y\})^ {\nu+\color{red}\mathbf{e_i}}$};
  \draw [] (v12) -- (v1);
  \draw [] (v12) -- (v2);
  \draw [] (v2) -- (v0); 
  \draw [] (v1) -- (v0);
  \node at (0,-1) {Type I};
 \end{scope}

\begin{scope}[xshift=240]
\tikzstyle{every node}=[ inner sep=0pt, scale=0.80, minimum width=4pt]
  \node (u0) at (0,0)  {$A^{\nu}$};
  \node (uxyij) at (0,4)  {$(A\cup\{ x,y\})^ {\nu+\color{red}\mathbf{e_i}+\color{blue}\mathbf{e_j}}$};
  \node (uxi) at (-4,2)  {$(A\cup\{ x\})^ {\nu+\color{red}\mathbf{e_i}}$};
  \node (uxj) at (-1.5,2) {$(A\cup\{ x\})^ {\nu+\color{blue}\mathbf{e_j}}$};
  \node (uyi) at (1.5,2)  {$(A\cup\{ y\})^ {\nu+\color{red}\mathbf{e_i}}$};
  \node (uyj) at (4,2)  {$(A\cup\{ y\})^ {\nu+\color{blue}\mathbf{e_j}}$};

\draw (u0) -- (uxi);
\draw (u0) -- (uxj);
\draw (u0) -- (uyi);
\draw (u0) -- (uyj);

\draw (uxyij) -- (uxi);
\draw (uxyij) -- (uxj);
\draw (uxyij) -- (uyi);
\draw (uxyij) -- (uyj);

\node at (0,-1) {Type II};

\end{scope}
\end{tikzpicture}
\end{center}
\caption[]{Intervals of length $2$ in $\B_n^w$}\label{figure:cohomologyrelations}
\end{figure}

\begin{theorem}\label{theorem:cohrelations}
 The set $\{ \bar{\c}(\sigma) \mid  \sigma \in \sym_{\mu}\}$ is a generating set for 
$\widetilde
H^{n-2}((\hat 0, [n]^{\mu}))$,
subject only to the relations for $i \ne j \in \supp(\mu)$

\begin{align}&\bar{\c}(\alpha x^i y^i \beta) + \bar{\c}(\alpha y^i x^i 
\beta)=0 \label{relation:1h}\\
&\bar{\c}(\alpha x^i y^j \beta) + \bar{\c}(\alpha y^j x^i 
\beta)+ \bar{\c}(\alpha y^i x^j \beta) + \bar{\c}(\alpha x^j y^i 
\beta)= 0\label{relation:2h}
\end{align}

\end{theorem}

\begin{proof}
By the comments above we know that $\{\bar{\c}(\sigma)\,\mid\,\sigma \in \sym_{\mu}\}$ is 
a set of generators in $\MM((\hat 0, [n]^{\mu}))$. Observe that the relations 
(\ref{relation:1h}) and (\ref{relation:2h}) correspond exactly to the cohomology relations of Type I 
and Type II respectively and these generate the space of cohomology relations.
\end{proof}

 \subsection{The isomorphism}
Following relations (\ref{relation:colantisym1}) and (\ref{relation:colantisym2}) we 
can conclude a similar proposition for $\L(\mu)$.

 \begin{proposition}\label{proposition:lrelations}
  The set $\{ \wedge(\sigma) \mid  \sigma \in \sym_{\mu}\}$ is a generating set for 
$\L(\mu)$ subject only to the relations for $i \ne j \in \supp(\mu)$

\begin{align}&\wedge(\alpha x^i y^i \beta) + \wedge(\alpha y^i x^i 
\beta)=0 \label{relation:1}\\
&\wedge(\alpha x^i y^j \beta) + \wedge(\alpha y^j x^i 
\beta)+ \wedge(\alpha y^i x^j \beta) + \wedge(\alpha x^j y^i 
\beta)= 0\label{relation:2}
\end{align}
  
 \end{proposition}

\begin{theorem}\label{theorem:explicitisomorphism}
 For each $\mu \in \wcomp_{n}$,  the map
$\varphi:\L(\mu)\rightarrow \widetilde H^{n-2}((\hat 0, [n]^{\mu}))$ determined by
\[
\varphi(\wedge(\sigma))=\bar{\c}(\sigma) \text{ for all }\sigma \in \sym_{\mu},\]
is an $\sym_n$-module isomorphism.
\end{theorem}
 
\begin{proof}
The generators of the two sets $\L(\mu)$ and $\widetilde H^{n-2}((\hat 0, [n]^{\mu}))$ are indexed 
by colored permutations in $\sym_{\mu}$ and $\varphi$ maps generators to generators. 
By Theorem \ref{theorem:cohrelations} and Proposition \ref{proposition:lrelations}, $\varphi$ also 
maps relations to relations and clearly respects the $\sym_n$ action.
\end{proof}

\section{Homotopy type of maximal intervals in $\B_n^w$}\label{section:homotopytype}

\subsection{EL-labeling}\label{section:ellabeling}

Let $P$ be a bounded poset. Recall from Section~\ref{section:isomorphism} that an edge labeling is a 
map $\bar 
\lambda: \mathcal E(P) \to \Lambda$,  from the set $\mathcal E(P)$ of edges of the 
Hasse diagram $P$ to some fixed poset $\Lambda$. Recall also that to any maximal chain $c = (\hat 0 
= x_0 \lessdot x_1 \lessdot \cdots \lessdot x_{\ell-1} \lessdot x_{\ell}= \hat 1)$ in $P$ 
corresponds a word of labels
$\bar \lambda(c) = \bar \lambda(x_0, x_1) \bar \lambda(x_1, x_2) \cdots \bar \lambda(x_{\ell-1}, 
x_{\ell}).$
We say that  $c$ is  \emph{increasing} if its word of labels $\bar \lambda(c)$ is
\emph{strictly}  increasing, that is, $c$ is  increasing if 
\begin{align*}
 \bar \lambda(x_0, x_1) < \bar \lambda(x_1, x_2)<  \cdots < \bar \lambda(x_{\ell-1}, 
x_{\ell}). 
\end{align*}
  We say 
that  $c
$ is  \emph{ascent-free} if its word of labels $\bar \lambda(c)$ has no 
ascents,
i.e. $\bar \lambda(x_i, x_{i+1}) \not<  \bar \lambda(x_{i+1}, x_{i+2}) $, for all 
$i=0,\dots,\ell-2$.
\emph{ An edge-lexicographical
labeling} (EL-labeling, for short)  of
$P$ is an edge labeling such that in each closed
interval $[x,y]$ of $P$, there is a unique  increasing maximal chain, and this chain
lexicographically precedes all other maximal chains of $[x,y]$. See \cite{Bjorner1980} and 
\cite{BjornerWachs1983} for more information about EL-labelings.

We let $\Lambda=[n+1]\times\PP$ be the product poset of the totally ordered sets $[n+1]$ 
and $\PP$ and we define for any  $S\subseteq \PP$ the labeling $\bar 
\lambda:\E(\widehat{\B_n^S})\rightarrow [n+1]\times\PP$ 
by
\begin{align}\label{definition:ellabelinghat}
 \bar\lambda(A^{\nu},(A\cup \{x\})^{\nu+\mathbf{e_i}})&=x^i,\\
 \bar\lambda([n]^{\mu},\hat{1})&=(n+1)^1.\nonumber
\end{align}
In Figure \ref{figure:weightedbooleanposetn3k2EL} this labeling is illustrated in the case of 
$\widehat{\B_n^{[2]}}$. The edges have been differentiated by shape and color 
corresponding to the different 
labels that appear in the legend. Note that this labeling restricts to the labeling of equation
(\ref{definition:labeling}) in the maximal intervals $[\hat{0},[n]^{\mu}]$.
\begin{figure}[ht]
\begin{center} 
\begin{tikzpicture}[scale=0.6]

\tikzstyle{style11}=[color=black,thin]
\tikzstyle{style21}=[color=blue,dashed]
\tikzstyle{style31}=[color=violet,dashdotted]
\tikzstyle{style12}=[color=red,thick]
\tikzstyle{style22}=[color=green,double]
\tikzstyle{style32}=[color=orange,dotted,thick]
\tikzstyle{style1}=[color=yellow,dash pattern=on 10pt off 6pt,thick]

\tikzstyle{every node}=[inner sep=0pt, scale=0.7, minimum width=4pt]

\node (nempty) at (0,0) {$\emptyset^{(0,0)}$};

\node (n1-10) at (-5,3) {$1^{(1,0)}$};
\node (n2-10) at (-3,3) {$2^{(1,0)}$};
\node (n3-10) at (-1,3) {$3^{(1,0)}$};
\node (n1-01) at (1,3) {$1^{(0,1)}$};
\node (n2-01) at (3,3) {$2^{(0,1)}$};
\node (n3-01) at (5,3) {$3^{(0,1)}$};

\node (n12-20) at (-8,6) {$12^{(2,0)}$};
\node (n13-20) at (-6,6) {$13^{(2,0)}$};
\node (n23-20) at (-4,6) {$23^{(2,0)}$};
\node (n12-11) at (-2,6) {$12^{(1,1)}$};
\node (n13-11) at (0,6) {$13^{(1,1)}$};
\node (n23-11) at (2,6) {$23^{(1,1)}$};
\node (n12-02) at (4,6) {$12^{(0,2)}$};
\node (n13-02) at (6,6) {$13^{(0,2)}$};
\node (n23-02) at (8,6) {$23^{(0,2)}$};

\node (n123-30) at (-6,9) {$123^{(3,0)}$};
\node (n123-21) at (-2,9) {$123^{(2,1)}$};
\node (n123-12) at (2,9) {$123^{(1,2)}$};
\node (n123-03) at (6,9) {$123^{(0,3)}$};

\node (n1) at (0,12) {$\hat{1}$};

\draw[style11] (nempty)--(n1-10);
\draw[style21] (nempty)--(n2-10);
\draw[style31] (nempty)--(n3-10);
\draw[style12] (nempty)--(n1-01);
\draw[style22] (nempty)--(n2-01);
\draw[style32] (nempty)--(n3-01);

\draw[style21] (n1-10)--(n12-20);
\draw[style22] (n1-10)--(n12-11);
\draw[style31] (n1-10)--(n13-20);
\draw[style32] (n1-10)--(n13-11);

\draw[style11] (n2-10)--(n12-20);
\draw[style12] (n2-10)--(n12-11);
\draw[style31] (n2-10)--(n23-20);
\draw[style32] (n2-10)--(n23-11);

\draw[style21] (n3-10)--(n23-20);
\draw[style22] (n3-10)--(n23-11);
\draw[style11] (n3-10)--(n13-20);
\draw[style12] (n3-10)--(n13-11);

\draw[style22] (n1-01)--(n12-02);
\draw[style21] (n1-01)--(n12-11);
\draw[style32] (n1-01)--(n13-02);
\draw[style31] (n1-01)--(n13-11);

\draw[style12] (n2-01)--(n12-02);
\draw[style11] (n2-01)--(n12-11);
\draw[style32] (n2-01)--(n23-02);
\draw[style31] (n2-01)--(n23-11);

\draw[style22] (n3-01)--(n23-02);
\draw[style21] (n3-01)--(n23-11);
\draw[style12] (n3-01)--(n13-02);
\draw[style11] (n3-01)--(n13-11);

\draw[style31] (n12-20)--(n123-30);
\draw[style32] (n12-20)--(n123-21);
\draw[style21] (n13-20)--(n123-30);
\draw[style22] (n13-20)--(n123-21);
\draw[style11] (n23-20)--(n123-30);
\draw[style12] (n23-20)--(n123-21);

\draw[style31] (n12-11)--(n123-21);
\draw[style32] (n12-11)--(n123-12);
\draw[style21] (n13-11)--(n123-21);
\draw[style22] (n13-11)--(n123-12);
\draw[style11] (n23-11)--(n123-21);
\draw[style12] (n23-11)--(n123-12);

\draw[style31] (n12-02)--(n123-12);
\draw[style32] (n12-02)--(n123-03);
\draw[style21] (n13-02)--(n123-12);
\draw[style22] (n13-02)--(n123-03);
\draw[style11] (n23-02)--(n123-12);
\draw[style12] (n23-02)--(n123-03);

\draw[style1] (n1)--(n123-12);
\draw[style1] (n1)--(n123-03);
\draw[style1] (n1)--(n123-30);
\draw[style1] (n1)--(n123-21);

\tikzstyle{every node}=[blue, inner sep=0pt, scale=0.6, minimum width=4pt]

\node at (-9,2.5) {$1^1$};
\draw[style11] (-8.5,2.5)--(-7,2.5);
\node at (-9,2) {$2^1$};
\draw[style21] (-8.5,2)--(-7,2);
\node at (-9,1.5) {$3^1$};
\draw[style31] (-8.5,1.5)--(-7,1.5);
\node at (-9,1) {$1^2$};
\draw[style12] (-8.5,1)--(-7,1);
\node at (-9,0.5) {$2^2$};
\draw[style22] (-8.5,0.5)--(-7,0.5);
\node at (-9,0) {$3^2$};
\draw[style32] (-8.5,0)--(-7,0);
\node at (-9,-.5) {$4^1$};
\draw[style1] (-8.5,-0.5)--(-7,-0.5);
\end{tikzpicture}
\end{center}
\caption[]{EL-labeling in $\widehat{\B_3^{[2]}}$}\label{figure:weightedbooleanposetn3k2EL}
\end{figure}
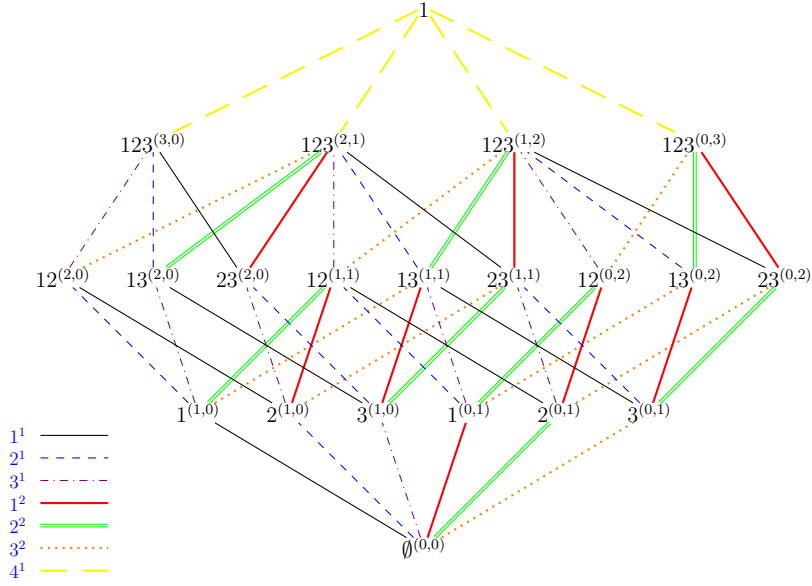

 \begin{theorem}\label{theorem:ellabelingposet}
 The labeling $\bar \lambda:\E(\widehat{\B_{n}^S})\rightarrow [n+1]\times\PP$ in 
(\ref{definition:ellabelinghat}) is an
EL-labeling of $\widehat{\B_{n}^S}$.  
\end{theorem}

\begin{proof}
We want to show that in every closed interval of $\widehat{\B_{n}^S}$  there is a unique 
increasing chain (from bottom to top), which is also lexicographically first. 
Note that any interval $[A^{\nu},(A\cup B)^{\nu+\mu}]$ in $\B_n^{S}$ is canonically isomorphic to 
an interval $[\hat{0},B^{\mu}]$ in $\B_B^{S}$, where $\B_B$ is the boolean algebra on the set 
$B \subseteq [n]$, and this isomorphism respects the labeling $\bar \lambda$. So we only need to 
care for intervals of the form $[\hat{0},[n]^{\mu}]$ and of the form $[\hat{0},\hat{1}]$.

For Intervals of the form $[\hat{0},[n]^{\mu}]$ there is only one possible increasing 
label word $1^{u_1}2^{u_2}\cdots n^{u_{n}}$ with $u_1\le u_2\le \cdots \le u_n$. This label word is 
lexicographically first and since, by Proposition 
\ref{proposition:bijectionchainscoloredpermutations}, we know that $\sym_{\mu}$ is in bijection 
with 
the set of maximal chains of $[\hat{0},[n]^{\mu}]$ only one chain has this label word, that is  
 \[
  \hat{0}\lessdot \{1\}^{u_1}\lessdot \{1,2\}^{u_1+u_2} \lessdot \cdots \lessdot 
[n]^{\sum_{i=1}^n u_i}.
 \] 

For the interval $[\hat 0,\hat 1]$  an increasing chain $c$ must be of the form $\\c^\prime \cup 
\{\hat 1\}$, where  $\c^\prime$ is the unique increasing chain of
some interval $[\hat{0},[n]^{\mu}]$. All such chains will have last step $n^u(n+1)^1$ and this is 
only increasing when $u=1$, so $\c^\prime$ is the increasing chain in $[\hat{0}, 
[n]^{n\mathbf{e_1}}]$. This unique increasing chain has word of labels $1^12^1\cdots (n+1)^1$ that 
is clearly 
lexicographically first.
\end{proof}

The following theorem links lexicographic shellability with topology.

\begin{theorem}[Bj\"orner and Wachs \cite{BjornerWachs1996}] \label{theorem:elth}
Let $\bar \lambda$ be an EL-labeling  of a bounded   poset $P$. 
Then for all $x<y$ in $P$, 
\begin{enumerate}
\item the open interval $(x,y)$ is homotopy equivalent to a wedge of  spheres, where for each $r \in
\NN$ the number of spheres  of dimension $r$ is the number of ascent-free maximal chains of the
closed interval $[x,y]$ of length $r+2$. 
\item the set
$$\{\bar c \mid c \mbox{ is an ascent-free maximal chain of $[x,y]$ of length } r+2 \}$$
forms a basis for cohomology $\widetilde H^{r}((x,y))$, for all $r$.
\end{enumerate}
 \end{theorem}

We would like now to characterize the ascent-free chains 
of the EL-labeling of Theorem 
\ref{theorem:ellabelingposet}. We already know  by 
Proposition 
\ref{proposition:bijectionchainscoloredpermutations} that the 
maximal chains in $\B_n^S$ are in bijection with permutations in $\sym_n^{S}$. Since any maximal 
chain in $\widehat{\B_n^S}$ is of the form $\c^\prime \cup \{\hat{1}\}$ where $\c^\prime$ is a 
maximal chain in $\B_n^S$ then the permutations in $\sym_n^S$ are also in bijection with maximal 
chains in $\widehat{\B_n^S}$. Recall that the set $\ninc_n$ is the set of nonincreasing colored 
permutations. For $\sigma \in \sym_n^S$ denote $\hat{\c}(\sigma):=\c(\sigma)\cup 
\{\hat{1}\}$ and denote by $\overline{\ninc_n^S}$ the set of permutations in $\ninc_n^S$ with 
$\clr(\sigma(n))\neq 1$. We have the following characterization of the ascent-free chains.

\begin{theorem}\label{thm:ascfreeEL}
The set $\{\c(\sigma) \mid \, \sigma \in \ninc_{\mu}\}$ is the set of 
ascent-free maximal chains of $[\hat 0, [n]^\mu]$ and 
the set $\{\hat{\c}(\sigma) \mid \, \sigma \in \overline{\ninc_n^S}\}$ is the set of 
ascent-free maximal chains of $\widehat{\B_n^S}$ in the EL-labeling of Theorem
\ref{theorem:ellabelingposet}.
\end{theorem}
\begin{proof}
An ascent in the word of labels of a maximal chain $\c$ is of the form 
$$x^i=\bar\lambda(A^{\nu},(A\cup\{x\})^{\nu+\mathbf{e_i}})<\bar 
\lambda((A\cup\{x\})^{\nu+\mathbf{e_i}},(A\cup\{x,y\})^{\nu+\mathbf{e_i}+\mathbf{e_j}})=y^j,$$
when $x<y$ and $i\le j$.  This 
corresponds exactly to the definition  of an ascent $x^iy^j$ in a colored permutation. Using 
Proposition \ref{proposition:bijectionchainscoloredpermutations} we see that ascent-free chains in 
$[\hat{0},[n]^{\mu}]$ correspond to colored permutations with no ascents. To describe the 
ascent-free chains 
in $\widehat{\B_n^S}$ we now are only missing to check that there is no ascent of the form
$$x^i=\bar
\lambda([n]\setminus\{x\}^{\mu-\mathbf{e_i}},[n]^{\mu})<\bar 
\lambda([n]^{\mu},\hat{1})=(n+1)^1,$$ which will happen exactly in the case where
$i=1$.
\end{proof}

We obtain the following corollaries of Theorems \ref{theorem:ellabelingposet}, \ref{theorem:elth},
\ref{thm:ascfreeEL} and the isomorphism of Theorem \ref{theorem:isomorphism}.

\begin{corollary}
 The poset $\widehat{\B_n^S}$ is Cohen-Macaulay and its order complex 
$\Delta(\B_n^S \setminus \{\hat{0}\})$ has the homotopy type of a wedge of $|\overline{\ninc_n^S}|$
spheres of dimension $(n-1)$. For every $\mu \in \wcomp$ the interval $[\hat{0},[n]^{\mu}]$ is 
Cohen-Macaulay and its order complex $\Delta((\hat{0},[n]^{\mu}))$ has the homotopy type of a wedge 
of $|\ninc_{\mu}|$ spheres of dimension $(n-2)$.
\end{corollary}

\begin{corollary}
 The set $\{\c(\sigma)\setminus\{\hat{0}\}\,\mid\, \sigma \in \overline{\ninc_n^S}\}$ is a basis 
for $\widetilde H^{n-1}(\B_n^S \setminus \{\hat{0}\})$. For  every $\mu \in \wcomp$ the set  
$\{\bar\c(\sigma)\,\mid\, \sigma \in \ninc_{\mu}\}$ is a basis for $\widetilde 
H^{n-2}((\hat{0},[n]^{\mu}))$.
\end{corollary}

\begin{corollary}\label{corollary:dimensionlmu}
For  every $\mu \in \wcomp$ the set  
$\{\wedge(\sigma)\,\mid\, \sigma \in \ninc_{\mu}\}$ is a basis for $\L(\mu)$. Consequently,
$\dim \L(\mu)=|\ninc_{\mu}|$.
\end{corollary}

\begin{proof}[Proof of Theorem \ref{theorem:dimensionstype}]
 If $\sigma \in \sym_n^{\PP}$ is a colored permutation, denote by $\tilde\sigma \in \sym_n$ the 
underlying uncolored permutation associated to $\sigma$. For example if 
$\sigma={\color{blue}2^1}{\color{orange}1^4}{\color{red}3^2}$ then $\tilde\sigma=213$. Note that 
the type $\lambda(\tau)$ defined in Section~\ref{section:mainresults} for a permutation $\tau \in 
\sym_n$ is closely related to the coloring condition in $\ninc_n$. If $\sigma \in  \ninc_n$ such 
that $\tilde\sigma=\tau$ the colors in each part of the partition $\pi(\tau)$ need to strictly 
decrease from left to right. If $B$ is a block of $\pi(\tau)$ of size $|B|=i$ then the elementary 
symmetric function $e_i(\xx)$ enumerates all the possible ways of coloring the letters in $B$. 
Then the contribution to the generating function (\ref{definition:euleriansymmetric}) of all the 
nonincreasing colored permutations with underlying uncolored permutation $\tau$ is 
$e_{\lambda(\tau)}(\xx)$. By Corollary \ref{corollary:dimensionlmu} and the comments above we have
\begin{align*}
 \sum_{\mu \in \wcomp_n}\dim \L(\mu)\xx^{\mu}&=\sum_{\mu \in \wcomp_n}|\ninc_{\mu}|\xx^{\mu}\\
 &=\sum_{\sigma \in \ninc_n}\xx^{\mu(\sigma)}\\
 &=\sum_{\tau \in \sym_n}\sum_{\substack{\sigma \in \ninc_n\\\tilde\sigma=\tau}}\xx^{\mu(\sigma)}\\
  &=\sum_{\tau \in \sym_n}e_{\lambda(\tau)}(\xx)\qedhere
\end{align*}
\end{proof}

\section{The Frobenius characteristic of $\L(\mu)$}\label{section:frobeniuscharacteristic}
To prove Theorem \ref{theorem:representationmultiplicativeinverse} we will use a technique 
introduced by Sundaram \cite{Sundaram1994} (see also \cite{Wachs1999}) to compute 
group representations on the (co)homology of Cohen-Macaulay posets. This technique uses the concept 
of Whitney (co)homology that was introduced by Baclawski in \cite{Baclawski1975}.
For information not presented here about symmetric functions  and the representation theory of the 
symmetric group see \cite{Macdonald1995}, \cite{Sagan2001}, \cite{JamesKerber1981} and 
\cite[Chapter 7]{Stanley1999}.

\subsection{A multiplicative inverse formula}
In $\mathnormal{Rep}_{\sym}$, the product is called the \emph{induction 
product} and is defined for an $\sym_m$-module $V$ and an $\sym_n$-module $W$ by 
$$V\circ W:=(V \otimes W) \uparrow_{\sym_m \times \sym_n}^{\sym{m+n}},$$
where $\uparrow_*^*$ denotes induction.
We will need the following proposition that is one of the main properties that makes the 
Frobenius characteristic map $\ch$ 
 a ring isomorphism.

\begin{proposition}[\cite{Macdonald1995}]\label{proposition:tensorch}
Let $V$ be an $\sym_m$-module and $W$ an $\sym_n$-module. Then
\begin{align*}
 \ch \left ( (V \otimes W) \uparrow_{\sym_m \times \sym_n}^{\sym_{m+n}} \right )&= \ch V\ch W.
\end{align*}
\end{proposition}

\emph{Whitney cohomology} (over the field ${\bf k}$) of a poset $P$ with a minimum element $\hat 
0$ 
can be
defined for each integer $r$ as follows:
\begin{align*}
 WH^r(P):= \bigoplus_{x \in P} \widetilde H^{r-2}((\hat 0, x);{\bf k}).
\end{align*}
In the case of a pure Cohen-Macaulay poset this formula becomes
\begin{align}\label{equation:defwhitneyhomology}
 WH^r(P)= \bigoplus_{\substack{x \in P\\ \rho(x)=r}} \widetilde H^{r-2}((\hat 0, x);{\bf k}).
\end{align}
If a group $G$ of automorphisms acts on the poset $P$, this action induces a 
representation of $G$  
on $WH^r(P)$ for every $r$. From equation 
(\ref{equation:defwhitneyhomology}), when $P$ is pure and Cohen-Macaulay, $WH^r(P)$ breaks into the 
direct 
sum of 
$G$-modules
\begin{align}\label{equation:defwhitneyhomologyinduced}
 WH^r(P)\cong_{G} \bigoplus_{\substack{x \in P/\sim\\ \rho(x)=r}} \widetilde H^{r-2}((\hat 0, 
x);{\bf 
k}) \bigl \uparrow _{G_x}^G,
\end{align}
where  $P/\sim$ is a set of orbit representatives and $G_x$ the stabilizer of $x$. 
The following result of Sundaram \cite{Sundaram1994} can be used to compute characters of  
$G$-representations on the (co)homology of pure $G$-posets.
\begin{lemma}[\cite{Sundaram1994} Lemma 1.1]\label{lemma:sundaramsum}
Let $P$ be a bounded poset of length $\ell\ge 1$ and let $G$ be a group of automorphisms of $P$.
Then the following isomorphism of virtual $G$-modules holds
\begin{align*}
 \bigoplus_{i=0}^{\ell}(-1)^{\ell-i}WH^i(P)\cong_{G}0.
\end{align*}
\end{lemma}

We know by Theorem \ref{theorem:ellabelingposet} that for every $\mu \in \wcomp_n$ the poset 
$[\hat{0},[n]^{\mu}]$ is Cohen-Macaulay. We will apply Lemma \ref{lemma:sundaramsum} to 
$[\hat{0},[n]^{\mu}]^{*}$, that is the dual poset of $[\hat{0},[n]^{\mu}]$.

Now we specify a set of orbit representatives for the action of $\sym_n$ on 
$[\hat{0},[n]^{\mu}]^{*}$.
Denote by $\alpha_{\eta}$, the weighted subset $[|\eta|]^{\eta}$ of $[n]$ where $\eta \in \wcomp$ 
is such that $|\eta|\le n$.
We claim that 
$$\{\alpha_{\eta} \mid \eta \in \wcomp \text{ and }  \eta \le \mu \}$$ 
is such a set of orbit representatives. 
To see this, note that any weighted subset $\beta \in 
[\hat{0},[n]^{\mu}]^{*}$ can be obtained as $\beta=\sigma \alpha_{\eta}$ for suitable $\eta \in 
\wcomp$ such that $\eta \le \mu$.
It is also clear that $\alpha_{\eta} \ne \sigma 
\alpha_{\eta^{\prime}}$ for $\eta 
\ne \eta^{\prime}$ and for every $\sigma \in \sym_n$.
The weighted subset $\alpha_{\eta}$ has  the Young subgroup
$\sym_{|\eta|}\times \sym_{n-|\eta|}$ as stabilizer.

Applying equation (\ref{equation:defwhitneyhomologyinduced}) to $[\hat{0},[n]^{\mu}]^{*}$ we obtain,

\begin{align}\label{equation:whitneydualinterval}
  WH^r([\hat{0},[n]^{\mu}]^{*})\cong_{\sym_n}\bigoplus_{\substack{\eta \in \wcomp \\\eta \le \mu \\ 
|\eta|=n-r}} 
\widetilde 
{H}^{r-2}((\alpha_{\eta},[n]^{\mu}))\bigl 
\uparrow_{\sym_{|\eta|}\times \sym_{n-|\eta|}}^{\sym_{n}}.
\end{align}

Note that when  $r=1$  the open interval $(\alpha_{\eta},[n]^{\mu})$  is the empty poset, 
hence $\widetilde {H}^{r-3}((\alpha_{\eta},[n]^{\mu}))$ is isomorphic to the trivial 
representation of $\sym_{n-1}\times \sym_{1}$.  When $r=0$, we have that
$\alpha_{\eta}=[n]^\mu$ and in this case we use the convention that $\widetilde 
{H}^{r-3}((\alpha_{\eta},[n]^{\mu}))$ is isomorphic to the trivial representation of 
$\sym_n$.

We apply Lemma \ref{lemma:sundaramsum} together with equation 
(\ref{equation:whitneydualinterval}) to obtain the following result.

\begin{lemma} \label{lemma:equivariantrecursivity}For $n\ge 0$ and $\mu \in \wcomp_{n}$ we have 
the following $\sym_n$-module isomorphism

\begin{align}\label{equation:equivariantrecursivity}
  \mathbf{1}_{\sym_n}\delta_{n,0} \cong_{\sym_n}\bigoplus_{\substack{\eta \in \wcomp \\ \eta \le
\mu}} (-1)^{|\eta|}\widetilde 
{H}^{n-|\eta|-2}((\alpha_{\eta},[n]^{\mu}))\bigl 
\uparrow_{\sym_{|\eta|}\times \sym_{n-|\eta|}}^{\sym_{n}},
\end{align}
where $\mathbf{1}_{\sym_n}$ denotes the trivial representation of $\sym_n$ and $\delta_{n,0}$  
the Kronecker delta function.

\end{lemma}

\begin{lemma}\label{lemma:isomodules}For all $n \ge 0$ and $\eta,\nu \in \wcomp$ with 
$|\nu|+|\eta|=n$, the following $\sym_{|\eta|}\times \sym_{|\nu|}$-module isomorphism 
holds:
\begin{align*}
\widetilde{H}^{|\nu|-2}((\alpha_{\eta},[n]^{\nu+\eta}))
\cong_{\sym_{|\eta|}\times \sym_{|\nu|}}
\mathbf{1}_{\sym_{|\eta|}}\otimes 
\widetilde{H}^{|\nu|-2}((\hat{0},[|\nu|]^{\nu})).
\end{align*}
\end{lemma}
\begin{proof} The poset $[\alpha_{\eta},[n]^{\nu+\eta}]$ is canonically isomorphic to a 
product poset $P \times Q$ where $P$ is the $\sym_{|\eta|}$-poset with the unique 
element $[|\eta|]^{\eta}$  and $Q$ is the interval $[\hat{0},([n]\setminus [|\eta|])^{\nu}]$ that 
is an $\sym_{[n]\setminus [|\eta|]}$-poset. Furthermore by identifying 
$\sym_{[n]\setminus [|\eta|]}$ with $\sym_{|\nu|}$ we have that $Q$ and $[\hat{0},[|\nu|]^{\nu}]$ 
are isomorphic $\sym_{|\nu|}$-posets. The isomorphism of the $\sym_{|\eta|}\times  
\sym_{|\nu|}$-posets $[\alpha_{\eta},[n]^{\nu+\eta}]$ and $P\times [\hat{0},[|\nu|]^{\nu}]$ induces 
the corresponding $\sym_{|\eta|}\times \sym_{|\nu|}$-isomorphism in 
cohomology. The Lemma follows after an application of K\"unneth's 
theorem and after noticing that the action of $\sym_{|\eta|}$ on $P$ is trivial.
\end{proof}

\begin{theorem}\label{theorem:cohomologyrepresentation}
 We have that 
\begin{align*}
 \sum_{n\ge 0} \sum_{\mu \in \wcomp_{n}}\ch  
\widetilde{H}^{n-2}((\hat{0},[n]^{\mu}))\,\xx^{\mu}=\Bigl 
(\sum_{n\ge 0}(-1)^n h_{n}(\xx)h_{n}(\yy)
\Bigr)^{-1}.
\end{align*}

\end{theorem}

\begin{proof}
We use the convention that $\widetilde {H}^{r}((\alpha_{\eta},[n]^{\mu}))=0$ for all $r$ whenever
$\alpha_{\eta} \nleq [n]^{\mu}$.
Applying the Frobenius characteristic map $\ch$ (in $\yy$ variables) to both sides of equation 
(\ref{equation:equivariantrecursivity}), multiplying by 
$\xx^{\mu}$ and summing over all $\mu \in \wcomp_{n}$ with  
$\supp(\mu)\subseteq [k]$ for a fixed $k\in \PP$ yields
\begin{align*}
\delta_{n,0}=&\sum_{\substack{\mu \in \wcomp_{n}\\\supp(\mu)\subseteq 
[k]}}\mathbf{x}^{\mu}
\ch \left( \bigoplus_{\substack{\eta \in \wcomp \\ \eta \le
\mu}} (-1)^{|\eta|}\widetilde 
{H}^{n-|\eta|-2}((\alpha_{\eta},[n]^{\mu}))\bigl 
\uparrow_{\sym_{|\eta|}\times \sym_{n-|\eta|}}^{\sym_{n}}\right )\\
=&\sum_{\substack{\mu \in \wcomp_{n}\\\supp(\mu)\subseteq 
[k]}}\mathbf{x}^{\mu}
 \sum_{\substack{\eta \in \wcomp \\ \eta \le
\mu}} (-1)^{|\eta|}\ch \left(\widetilde 
{H}^{n-|\eta|-2}((\alpha_{\eta},[n]^{\mu}))\bigl 
\uparrow_{\sym_{|\eta|}\times \sym_{n-|\eta|}}^{\sym_{n}}\right )\\
=& \sum_{\substack{\eta \in \wcomp \\ |\eta|\le n \\\supp(\eta)\subseteq 
[k]}} (-1)^{|\eta|}\sum_{\substack{\nu \in \wcomp_{n-|\eta|}\\\supp(\nu)\subseteq 
[k]}}\mathbf{x}^{\eta+\nu}\ch \left(\widetilde 
{H}^{|\nu|-2}((\alpha_{\eta},[n]^{\nu+\eta}))\bigl 
\uparrow_{\sym_{|\eta|}\times \sym_{|\nu|}}^{\sym_{n}}\right ).
\end{align*}
Using Lemma \ref{lemma:isomodules}, Proposition \ref{proposition:tensorch}, the fact that $\ch 
\mathbf{1}_{n}=h_{n}(\yy)$ and summing over all $n 
\ge 0$ we have

\begin{align*}
1=&\sum_{\substack{\eta \in \wcomp \\\supp(\eta)\subseteq 
[k]}} (-1)^{|\eta|}\sum_{\substack{\nu \in \wcomp\\\supp(\nu)\subseteq 
[k]}}\mathbf{x}^{\eta+\nu}\ch \left(\mathbf{1}_{\sym_{|\eta|}}\otimes 
\widetilde{H}^{|\nu|-2}((\hat{0},[|\nu|]^{\nu}))\bigl 
\uparrow_{\sym_{|\eta|}\times \sym_{|\nu|}}^{\sym_{|\eta|+|\nu|}}\right )\\ 
=&\sum_{\substack{\eta \in \wcomp \\\supp(\eta)\subseteq 
[k]}} (-1)^{|\eta|}\sum_{\substack{\nu \in \wcomp\\\supp(\nu)\subseteq 
[k]}}\mathbf{x}^{\eta+\nu}\ch \mathbf{1}_{\sym_{|\eta|}} \ch 
\widetilde{H}^{|\nu|-2}((\hat{0},[|\nu|]^{\nu}))\\ 
=&\sum_{\substack{\eta \in \wcomp \\\supp(\eta)\subseteq 
[k]}} (-1)^{|\eta|}h_{\sym_{|\eta|}}(\yy) \mathbf{x}^{\eta}\sum_{\substack{\nu \in 
\wcomp\\\supp(\nu)\subseteq 
[k]}}\ch 
\widetilde{H}^{|\nu|-2}((\hat{0},[|\nu|]^{\nu}))\mathbf{x}^{\nu}\\ 
=&\left(\sum_{n\ge 0} (-1)^{n}h_{n}(\yy)h_{n}(x_1,\dots,x_k)\right)\left 
(\sum_{n\ge0}\sum_{\substack{\nu \in 
\wcomp_n\\\supp(\nu)\subseteq 
[k]}}\ch 
\widetilde{H}^{n-2}((\hat{0},[n]^{\nu}))\mathbf{x}^{\nu}\right).
\end{align*}
Note now, from the sequence of equalities above, that for any $\mu \in \wcomp$ the coefficient of 
$\xx^{\mu}$ in both sides of these equalities of power series is the same and independent of the 
value of $k$ as long as $\supp(\mu)\subseteq [k]$. Hence, letting $k$ be large enough completes 
the proof.
\end{proof}

\begin{remark}
 Theorems \ref{theorem:isomorphism} and \ref{theorem:cohomologyrepresentation} yield Theorem 
\ref{theorem:representationmultiplicativeinverse} as a corollary.
\end{remark}

Recall that, for some ring $R$, the ring $\Lambda_R$ of symmetric functions is the subring 
of the power series ring 
$R[[\xx]]$ formed by power series of bounded degree that are invariant under 
the permutation of the variables. We denote by $\widehat{\Lambda_R}$ the ring of symmetric power 
series in $R[[\xx]]$, that is the completion of $\Lambda_R$  with respect to the valuation given by 
the degree.
Note that the power series in Theorems \ref{theorem:cohomologyrepresentation} and 
\ref{theorem:representationmultiplicativeinverse} belong to $\widehat{\Lambda_R}$ with
$R=\Lambda_{\QQ}$. Now consider the map 
$E_1:\widehat{\Lambda_R} \rightarrow R[[y]]$ defined by:
\begin{align*}
 E_1(p_i(\yy))=y\delta_{i,1}
\end{align*}
for $i \ge 1$ and extended multiplicatively, linearly and taking the corresponding limits to all of 
$\widehat{\Lambda_R}$. Note that $E_1$ is an algebra homomorphism, or \emph{specialization}, 
since $E_1$ is 
defined on generators. The effect of $E_1$ can be understood under the following 
definition of the 
Frobenius characteristic map. Let $V$ be a representation of $\sym_n$ and $\chi^V$ its character, 
then
 $$\ch(V)=\dfrac{1}{n!}\sum_{\sigma \in \sym_n} \chi^V(\sigma)p_{\gamma(\sigma)}(\yy),$$
where $\gamma(\sigma)$ is the cycle type of the permutation $\sigma \in \sym_n$. 

We have that
\begin{align*}
 E_1(\ch V) = \dfrac{1}{n!}\sum_{\sigma \in \sym_n} \chi^V(\sigma)E_1 
\left(p_{\gamma(\sigma)}(\yy) 
\right) =  \chi_V(id)\dfrac{y^n}{n!}
= \dim V \dfrac{y^n}{n!}.
\end{align*}
In particular since $h_n(\yy)=\ch(\mathbf{1}_{\sym_{n}})$, the Frobenius characteristic of 
the trivial representation of $\sym_n$, we have that
$E_1(h_n(\yy))=\dfrac{y^n}{n!}$.

\begin{proof}[Proof of Theorem \ref{theorem:multiplicativeinverse} ]
 Apply $E_1$ to Theorem \ref{theorem:representationmultiplicativeinverse}. 
\end{proof}

\subsection{Combinatorial interpretation of the multiplicative 
inverse}\label{subsection:gesselinterpretation}
We will use Theorem \ref{theorem:representationmultiplicativeinverse} to give an explicit formula 
for the representation of $\sym_n$ on $\L(\mu)$. To do this we will use a combinatorial 
interpretation of the multiplicative inverse of an 
ordinary generating function in terms of words with allowed and
forbidden links discovered in different but equivalent forms by Fr\"oberg 
\cite{Froberg1975}, Carlitz-Scoville-Vaughan \cite{CarlitzScovilleVaughan1976} and Gessel 
\cite{Gessel1977}. The theory outlined in \cite{Gessel1977} has a more general scope of application 
but it is equivalent to the simplified description that we present here.

We call a \emph{word} a finite string of elements of some partially ordered set $\A$ that we call
an \emph{alphabet} (and call its elements \emph{letters}). For 
example if $\A=\PP$ then $1222143$ and $2147$ are both examples of words. 
In particular any label by itself and the empty word $\emptyset$ are considered words. We denote by 
$\A^*$ the set of words with letters in $\A$. For any two $w_1, w_2 \in \A^*$ we define the 
\emph{product} $w_1w_2$ to be the word constructed by concatenation. 
For example if $w_1=13112$ and $w_2=316$ then $w_1w_2=13112316$. Note 
that the product is associative and so expressions like $w_1w_2\cdots w_k$
are well defined. Note that $\A^*$ together with concatenation is just the \emph{free monoid}
generated by $\A$ and for a given ring $R$ we denote by $R\langle \langle\A\rangle\rangle$ the ring 
of (noncommutative) power series on 
$\A$. To 
$w\in\A^*$ and $a\in \A$ we denote $m_a(w)$ the number of times the letter $a$ appears in $w$ and 
$|w|:=\sum_{a\in \A}m_a(w)$ the \emph{length} of $w$.
We call a \emph{link} the product of two letters. The set of links $\A^2$ is then in bijection with 
$\A \times \A$.

Consider a partition of the set $\A^2$ of links into two parts that we call the 
\emph{allowed links} $\LL(\A)$ and \emph{forbidden links} $\overline{\LL(\A)}$. Let $\W$ be the 
set of words in $\A^*$ constructed exclusively by the concatenation of allowed 
links and let $\overline{\W}$ the ones constructed using only forbidden links. For example 
if $\LL(\PP)$ is the set of links $ab$ whenever $a < b$ then $1234 \in \W$, $4221 \in \overline{\W}$ 
but $1324$ is neither in $\W$ nor in $\overline{\W}$. 
In particular, we consider the empty word (that we identify with $1 \in R$) and the letters in $\A$ 
as if they 
are both in $\W$ and $\overline{\W}$. 

Consider the following generating functions in $R \langle \langle\A \rangle\rangle$
\begin{align*}
 F(\W)&=\sum_{w \in \W}w,\\
  \overline{F}(\overline{\W})&=\sum_{w \in \overline{\W}}(-1)^{|w|}w.
  \end{align*}
The reader can consult \cite{Gessel1977} for a proof of the following theorem.
\begin{theorem}[c.f.\cite{Gessel1977}]\label{theorem:gessel}
In $R \langle \langle\A \rangle\rangle$, we have
\[
  F(\W)\overline{F}(\overline{\W})=1.
\]
\end{theorem}
\subsection{Explicit description of the $\sym_n$ representation on $\L(n)$}

A composition $\alpha=(\alpha(1),\dots,\alpha(\ell))$ is a finite sequence of elements $\alpha(i) 
\in \PP$. We say that $\alpha$ is a \emph{composition of $n$} if $|\alpha|:=\sum_{i}\alpha(i)=n$. 
We 
denote $\comp$ the set of compositions and $\comp_n$ the set of compositions of $n$. We denote 
$\lambda(\alpha)$ the integer partition obtained from $\alpha$ by reordering its parts in weakly 
decreasing order.

For (perhaps empty) integer partitions $\nu$ and $\lambda$ such that $\nu\subseteq \lambda$ 
(that is $\nu(i) \le \lambda(i)$ for all $i$), let
$S^{\lambda / \nu}$ denote the Specht module of skew shape $\lambda / \nu$ and $s_{\lambda / \nu}$ 
the 
Schur function of shape  $\lambda / \nu$. Recall that $s_{\lambda / \nu}$ is the image in the ring 
of symmetric functions of the Specht module $S^{\lambda / \nu}$  under the Frobenius 
characteristic map $\ch$, i.e., $\ch S^{\lambda / \nu}=s_{\lambda / \nu}$.

A \emph{skew hook} is a connected skew shape that avoids the shape $(2,2)$. Every skew hook can be 
described by a composition $\alpha$ whose parts are the lengths of the horizontal steps from left 
to right. We denote by $H(\alpha)$ the skew hook determined by $\alpha \in \comp$. 
See Figure \ref{figure:exampleskewhook} for an example of the skew hook associated to the 
composition $\alpha=(3,2,1,1,3)$. The Specht modules of skew hook shape are frequently called 
\emph{Foulkes 
representations} since they were studied by Foulkes in \cite{Foulkes1976}.

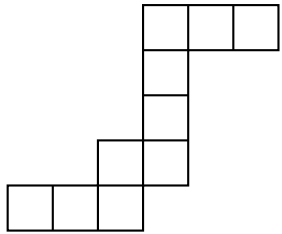
\begin{figure}[ht]
\begin{center} 
\begin{tikzpicture}[scale=0.6]
\tikzstyle{every path}=[thick];

\draw (0,0) -- (3,0)--(3,1)--(4,1)--(4,4)--(6,4)--(6,5)--(3,5)--(3,2)--(2,2)--(2,1)--(0,1)--cycle;
\draw (1,0) -- (1,1);
\draw (2,0) -- (2,1);
\draw (3,1) -- (2,1);
\draw (3,1) -- (3,2);
\draw (3,2) -- (4,2);
\draw (3,3) -- (4,3);
\draw (3,4) -- (4,4);
\draw (4,4) -- (4,5);
\draw (5,4) -- (5,5);

\end{tikzpicture}
\end{center}
\caption{Example of the skew hook corresponding to 
$\alpha=(3,2,1,1,3)$}\label{figure:exampleskewhook}
\end{figure}

\begin{theorem}[Gessel, personal communication]\label{theorem:explicitmultiplicativeinverse}
 We have that 
\begin{align*}
 \Bigl 
(\sum_{n\ge 0}(-1)^n h_{n}(\xx)h_{n}(\yy)
\Bigr)^{-1}=\sum_{n\ge 0} \sum_{\alpha \in \comp_{n}}e_{\lambda(\alpha)}(\xx)s_{H(\alpha)}(\yy).
\end{align*}
\end{theorem}
\begin{proof}
 Consider the alphabet $\A=\PP \times \PP$, that is, the set of biletters of the form $(a,b)$, 
with partial order given by $(a,b)\le(c,d)$ if and only if $a\le c$ and $b\le d$ (product order). 
Let $$\LL(\A)=\{(a,b)(c,d) \in \A^2\,\mid\,(a,b)\nleq(c,d)\}$$ so $\W$ is the set of 
words of the form
$$(a_1,b_1)\nleq(a_2,b_2)\nleq\cdots\nleq (a_n,b_n),$$ 
$\overline{\W}$ is the set of words of the form 
 $$(a_1,b_1)\le(a_2,b_2)\le\cdots\le(a_n,b_n)$$ and we have  by Theorem \ref{theorem:gessel} that 
$F(\W)\overline{F}(\overline{\W})=1$.

Now we consider the $\QQ$-algebra homomorphism 
(or specialization) $\xi:\QQ\langle \langle\A\rangle\rangle\rightarrow \QQ[[\xx,\yy]]$ defined on 
$\A^*$ by 
$\xi(w)=\prod_{i=1}^{|w|}x_{w_i(1)}y_{w_i(2)}$ where $w_i=(w_i(1),w_i(2))$ is the $i$th letter of 
$w$. Any word $w$ in $\overline{\W}$ of length $|w|=n$ can be uniquely constructed with a pair of 
weakly increasing sequences 
\begin{align*}
a_1\le a_2 \le\cdots\le a_n\\
b_1\le b_2\le\cdots\le b_n
\end{align*}

 and so
$$\xi(\overline{F}(\overline{\W}))=\sum_{n\ge 0}(-1)^n h_{n}(\xx)h_{n}(\yy).$$
Note that $(a_1,b_1)\nleq(a_2,b_2)$ if and only if either
\begin{enumerate}
 \item $b_1>b_2$ or
 \item $b_1\le b_2$ and $a_1>a_2$,
\end{enumerate}
so any word in $\W$ can be uniquely constructed with a pair of sequences:
\begin{align*}
 b_1\le\cdots \le b_{\alpha_1}&>&b_{\alpha_1+1}\le \cdots \le
b_{\alpha_1+\alpha_2}&>&\cdots&>&b_{\alpha_1+\cdots+\alpha_{\ell-1}+1}\le \cdots \le
b_{\alpha_1+\cdots+\alpha_{\ell}}\\
\underbrace{a_1>\cdots > a_{\alpha_1}}_{\alpha_1}&\,\,\,,& \underbrace{a_{\alpha_1+1}> \cdots >
a_{\alpha_1+\alpha_2}}_{\alpha_2}&\,\,\,,&\cdots&\,\,\,,&\underbrace{a_{\alpha_1+\cdots+\alpha_{
\ell-1 } +1 } > \cdots >
a_{\alpha_1+\cdots+\alpha_{\ell}}}_{\alpha_{\ell}}
\end{align*}
where $\alpha \in \comp$. 
Then
\begin{equation*}
\xi(F(\W))=\sum_{\alpha \in \comp}e_{\lambda(\alpha)}(\xx)s_{H(\alpha)}(\yy). \qedhere
 \end{equation*}
\end{proof}
\begin{remark}
 Note that essentially the same proof of Theorem \ref{theorem:explicitmultiplicativeinverse} 
presented above gives a more general noncommutative version of the identity that reduces to the one 
stated after applying the specialization that lets the variables within $\xx$ and within $\yy$ 
commute.
\end{remark}

If we number the cells of a skew hook $H$ left-to-right and bottom-to-top, $i$ is a \emph{descent} 
of $H$ if the cell $i+1$ is above the cell $i$. We denote by $\des(H)$ the descent set of $H$. For 
a standard Young tableau of shape $\lambda$ a \emph{descent} is an entry $i$ that is in a higher 
row than $i+1$.
\begin{proposition}[{c.f. \cite{Wachs2007}}] \label{proposition:skewhook}For a skew hook $H$ we 
have that
$$s_H(\yy)=\sum_{\lambda \vdash n}c_{H,\lambda}s_\lambda(\yy),$$
where $c_{H,\lambda}$ is the number of Young tableaux of shape $\lambda$ and descent set $\des(H)$.
\end{proposition}

\begin{remark}
Theorems \ref{theorem:multiplicativeinverse} and \ref{theorem:explicitmultiplicativeinverse} 
together with Proposition \ref{proposition:skewhook} yield Theorem \ref{theorem:representation} as a 
corollary.
\end{remark}

\section{The Koszul property}\label{section:Koszul}
A quadratic associative algebra $A$ and its Koszul dual (co)algebra $A^{\subfactorial}$ are 
said to be \emph{Koszul} if the 
\emph{Koszul complex} $A^{\subfactorial}\otimes_{\kappa}A$  is \emph{acyclic} (see 
\cite{LodayVallette2012} for the 
definitions). There 
are various techniques to conclude the Koszul property of an associative algebra. We use the 
technique in 
\cite{Mendez2010} that involves constructing a family of posets associated to $A$ and 
determining that these posets are Cohen-Macaulay. 
For a finite subset $S \in \PP$ the equivalence classes of colored permutations that generate 
$\S_S(n)$, considering the 
symmetry relations (\ref{relation:colsym1}) and (\ref{relation:colsym2}),   can be 
identified with \emph{colored subsets} $A^\mu$, where $A\subseteq [n]$ and $\mu \in \wcomp_{|A|}$ 
with $\supp(\mu)\in S$. For example 
${\color{blue}2^1}{\color{orange}1^4}$, ${\color{blue}1^1}{\color{orange}2^4}$, 
${\color{blue}2^4}{\color{orange}1^1}$ and ${\color{blue}1^4}{\color{orange}2^1}$ all represent 
the same generator in $\S_{[4]}(2)$, and we can represent this generator by $\{1,2\}^{(1,0,0,1)}$ 
since the underlying labels are $1$ and $2$ and exactly one of them has color $1$ and one has 
color $4$. Let $\overline \S_S(n)$ be the set of colored subsets whose underlying set 
is $A=[n]$, the map $[n]\mapsto \overline \S_S(n)$ defines a functor (Joyal's species 
\cite{Joyal1986}) from the 
category $\mathnormal{Set}$ of finite sets and bijections to the category $\mathfrak{F}$ of finite 
sets and arbitrary functions. It is not hard to 
verify that this functor actually defines a \emph{quadratic} (by relations (\ref{relation:colsym1}) 
and 
(\ref{relation:colsym2})) \emph{cancellative (injective) monoid or c-monoid} in the sense of 
\cite{Mendez2010}. Here given two disjoint sets $A$ and $B$, the multiplication 
$\mu_{A,B}:\overline\S_S(A)\times \overline\S_S(B) \rightarrow \overline\S_S(A\cup B)$ that gives 
the monoidal structure to $\overline\S_S$ is given by
$\mu_{A,B}(A^{\eta} \times B^{\nu})=(A\cup B)^{\eta +\nu}$ for $\eta \in \wcomp_{|A|}$ and $\nu 
\in \wcomp_{|B|}$.
Moreover, the family of posets associated to $\overline \S_S$ in the construction  
in 
\cite[Section 5.1]{Mendez2010} is precisely the family of posets $\B_n^S$ for $n\ge 0$ and 
the analytic functor $\S_S:\mathnormal{Vect}_{\kk}\rightarrow \mathnormal{Vect}_{\kk}$ associated 
to $\overline \S_S$ is $V\mapsto \S_S(V)$.
\begin{theorem}[{\cite[Proposition 22, Lemma 40, Theorem 41]{Mendez2010}}]\label{theorem:mendez} 
Let $\overline M$ be a $c$-monoid with associated analytic monoid $M$. Then $M$ and its 
Koszul dual $M^{!}$ are Koszul if and only if the maximal intervals of the poset associated to 
$\overline M$ are Cohen-Macaulay. 
\end{theorem}

Then the following theorem is a consequence of Theorems 
\ref{theorem:ellabeling} and \ref{theorem:mendez}.

\begin{theorem}\label{theorem:koszulity}
For a finite dimensional vector space $V$ and a finite subset $S \subseteq \PP$, the Koszul dual 
associative algebras $\S_S(V)$ and $\L_S(V)$ are Koszul.
\end{theorem}

\begin{remark}
 The theory developed in \cite{Mendez2010} can also be used to conclude the isomorphism of Theorem 
\ref{theorem:isomorphism}. This isomorphism is also a consequence of the fact that the family of 
posets $\B_n^S$ is the family associated to the c-monoid $\overline \S_S$. In Section 
\ref{section:isomorphism} of this work we went a bit in the longer direction by 
constructing and providing an explicit isomorphism between the cohomology of the posets associated 
to this particular c-monoid and its Koszul dual.
\end{remark}

\begin{remark}
 Note that we obtain as a corollary of  Theorem \ref{theorem:Smu_trivial} the equality of 
symmetric functions \begin{align*}
 \sum_{\mu \in \wcomp_{n}}\ch \S(\mu)\,\xx^{\mu}=h_n(\xx)h_n(\yy),
\end{align*}
and then Theorem \ref{theorem:representationmultiplicativeinverse} can be written as

\begin{align}\label{equation:koszulgeneratingidentity}
 \Bigl (\sum_{n\ge
0}(-1)^n \sum_{\mu \in \wcomp_{n}}\ch \S(\mu)\,\xx^{\mu}
\Bigr)\Bigl (\sum_{n\ge 0} \sum_{\mu \in \wcomp_{n}}\ch \L(\mu)\,\xx^{\mu}\ \Bigr )=1.
\end{align}
This is not a surprising coincidence since the version of equation 
(\ref{equation:koszulgeneratingidentity}) but with 
$x_i=1$ for all $i$ is a natural consequence of the Koszulness of $\S_S(V)$ and $\L_S(V)$ 
considering only the actions of the symmetric groups $\sym_n$ for all $n$ (see for example 
\cite{Mendez2010} or \cite{LodayVallette2012}). Equation (\ref{equation:koszulgeneratingidentity}) 
itself is also a consequence of the Koszulness of $\S_S(V)$ and $\L_S(V)$ if we consider 
a more general action. For all $k$, the general linear groups $\mathop{GL}_k$ act on the set of 
weak compositions of length $k$ and, since each of these actions commutes with the actions 
of the symmetric groups, then the spaces $\S_{[k]}(n)$ and $\L_{[k]}(n)$ are $\mathop{GL}_k 
\times \sym_n$-modules. We then consider the modified characteristic map $\ch$, from the set of 
polynomial $\mathop{GL}_k \times \sym_n$ representations to the set of symmetric functions in $\yy$ 
with symmetric polynomial coefficients in $(x_1,\dots,x_k)$, that in an irreducible $\mathop{GL}_k 
\times \sym_n$-module of the form $W \otimes V$, with $W$ an irreducible $\mathop{GL}_k$-module and 
$V$ an irreducible $\sym_n$-module, is defined by:
 $$\ch(W\otimes V):=\textstyle\mathop{char}_{\mathop{GL}_k}(W)\ch_{\sym_n}(V),$$
 where $\ch_{\sym_n}(V)$ is the Frobenius characteristic map considered before in $\yy$ variables 
and $\mathop{char}_{\mathop{GL}_k}(W)$ is the character map that assigns to a polynomial 
representation of $\mathop{GL}_k$ a symmetric polynomial in variables $x_1,x_2,\dots,x_k$.
\end{remark}

\section{Specializations}\label{section:specializations}

It turns out that Theorems \ref{theorem:representationmultiplicativeinverse} and
\ref{theorem:explicitmultiplicativeinverse} are doubly-symmetric generalizations of certain 
classical identities, including Euler's exponential formula for Eulerian polynomials.  

In Section \ref{section:frobeniuscharacteristic} from Theorem 
\ref{theorem:representationmultiplicativeinverse} using specialization $E_1$ we obtained Theorem 
\ref{theorem:multiplicativeinverse}. Together with Theorem \ref{theorem:dimensionstype} it can be  
written in the following form already obtained by the author in \cite{Dleon2016}.
\begin{theorem}[\cite{Dleon2016}]\label{theorem:exponentialmultiplicative}
 We have
 \begin{align*}
   \left(\sum_{n\ge0}(-1)^{n}h_{n}(\xx)\dfrac{y^n}{n!}\right )
^{-1}=\sum_{n\ge0}\left(\sum_{\sigma \in \sym_n}e_{\lambda(\sigma)}(\xx)\right)\dfrac{y^n}{n!},
 \end{align*}
\end{theorem}

The following Proposition is a standard result in the theory of symmetric functions and can be 
obtained as a special case of a more general Theorem of E\u gecio\u glu and Remmel 
in \cite{OmerRemmel1991}.

\begin{proposition}[{c.f. \cite[Theorem 2.3]{OmerRemmel1991}}]\label{proposition:htoe} For every 
$n\ge 
0$,
\begin{align*}
 h_n(\xx)=(-1)^n\sum_{\nu \in \comp_n}(-1)^{\ell(\nu)}e_{\nu}(\xx)
\end{align*}
\end{proposition}

Let $E_2:\Lambda \rightarrow \QQ[t]$ be the map defined by $E_2(e_i(\xx))=t$ for all $i \ge 1$ and 
$E_2(1)=1$. Since $E_2$ is defined on the generators $e_i$ it is immediate to check that 
$E_2$ is an algebra 
homomorphism or \emph{specialization}.
In the same manner it is also easy to verify that the specialization $E_2$ extends to a 
specialization $E_2:\Lambda[[y]] \rightarrow \QQ[t][[y]]$ in the algebra of power series in 
$y$ with symmetric function coefficients in $\Lambda$ (with variables $\xx$) defined by applying 
$E_2$ coefficientwise.
Note that for any $\lambda\vdash n$ we have that 
\begin{align}\label{equation:specializee}
E_2(e_{\lambda}(\xx))=t^{\ell(\lambda)}. 
\end{align}

\begin{lemma}For every $n \ge 1$,
\begin{align}\label{equation:specializeh}
 E_2(h_n(\xx))= t(t-1)^{n-1}.
\end{align}
\end{lemma}
\begin{proof}
Using Proposition \ref{proposition:htoe} and equation (\ref{equation:specializee}),
 \begin{align*}
  E_2(h_n(\xx))&=E_2\left((-1)^n\sum_{\nu \in 
\comp_n}(-1)^{\ell(\nu)}e_{\nu}(\xx)\right)\\
&=(-1)^n\sum_{\nu \in \comp_n}(-t)^{\ell(\nu)}\\
&=(-1)^n\sum_{k=1}^{n}\binom{n-1}{k-1}(-t)^{k}\\
&=(-1)^n(-t)(1-t)^{n-1},
 \end{align*}
 where $\binom{n-1}{k-1}$ is the number of compositions of $n$ into $k$ parts.
\end{proof}

Applying the specialization $E_2$ to the symmetric function
\begin{align*}
 \sum_{\sigma \in \sym_n}e_{\lambda(\sigma)}(\xx),
\end{align*}using 
equation 
(\ref{equation:specializee}) and the observation that 
$\ell(\lambda(\theta))=\des(\theta)+1$, where $\des(\sigma)=|\{i \in 
[n-1]\,\mid\,\sigma(i)>\sigma(i+1)\}|$, we obtain 
\begin{align*}
 E_2\left(\sum_{\sigma \in \sym_n}e_{\lambda(\sigma)}(\xx)\right)
 &=\sum_{\sigma \in \sym_n}t^{\ell(\lambda(\sigma))}\\
 &=\sum_{\sigma \in \sym_n}t^{\des(\sigma)+1}\\
&=A_{n}(t),
\end{align*}
the \emph{$n$-th Eulerian polynomial}.
The reader can verify now that by applying $E_2$ to the identity in Theorem 
\ref{theorem:exponentialmultiplicative} we obtain the classical Euler's exponential formula.
\begin{theorem}[\cite{Riordan1951}]\label{theorem:riordan} We have
 \begin{align*} 
   \dfrac{1-t}{1-te^{(1-t)y}}= \sum_{n\ge0}A_n(t)\dfrac{y^n}{n!}.
 \end{align*}
\end{theorem}

On the other hand if we apply the specialization $E_2$ directly to Theorem 
\ref{theorem:explicitmultiplicativeinverse} we obtain the equation
\begin{align}\label{equation:specializatione2only}
 \Bigl 
(1-h_1(\yy)+\sum_{n\ge 2}(-1)^n t(t-1)^{n-1}h_{n}(\yy)
\Bigr)^{-1}=\sum_{n\ge 0} \sum_{\alpha \in \comp_{n}}s_{H(\alpha)}(\yy)t^{\ell(\alpha)}.
\end{align}

The terms in the right-hand side of equation \ref{equation:specializatione2only} have a 
particular meaning. Indeed, the regular representation 
$\CC[\sym_n]$ of $\sym_n$ decomposes nicely into Foulkes representations.
\begin{theorem}[\cite{Foulkes1976}]\label{theorem:foulkes} We have that
 \begin{align*}
  h_1(\yy)^n=\ch(\mathbf{1}_n^{\otimes n} \uparrow_{\sym_1^{\times 
n}}^{\sym_n})=\ch\CC[\sym_n]= \sum_{\alpha \in \comp_{n}}s_{H(\alpha)}(\yy).
 \end{align*}
\end{theorem}

If we use Theorem \ref{theorem:foulkes} and specialize further in equation 
(\ref{equation:specializatione2only}) by setting $t=1$ we obtain 
\begin{align}\label{equation:specializatione2andt1}
 (1-h_1(\yy))^{-1}=\sum_{n\ge 0} \sum_{\alpha \in 
\comp_{n}}s_{H(\alpha)}(\yy)=\sum_{n\ge 0}\ch(\CC[\sym_n]),
\end{align}
indicating that the representation 
\begin{align*}
 \L_S(n)\simeq_{\sym_n}\bigoplus_{\substack{\mu \in \wcomp_n\\ \supp(\mu) 
\subseteq S}} \L(\mu)
\end{align*}
is an extension of the regular representation $\CC[\sym_n]$. In Figure 
\ref{figure:diagramspecializations} the reader can appreciate the relations between the different 
identities and specializations discussed in this section.

\begin{figure}[ht]
\everymath{\displaystyle}
\begin{tikzpicture}[scale=1.5]
\tikzstyle{every node}=[scale=0.75]
\node (a) at (3.5,5) {$\Bigl 
(\sum_{n\ge 0}(-1)^n h_{n}(\xx)h_{n}(\yy)
\Bigr)^{-1}=\sum_{n\ge 0} \sum_{\alpha \in \comp_{n}}e_{\lambda(\alpha)}(\xx)s_{H(\alpha)}(\yy)$ 
(Theorem \ref*{theorem:explicitmultiplicativeinverse})};
\node (b) at (8,3) {$\Bigl 
(\sum_{n\ge 0}(-1)^n h_{n}(\xx)\dfrac{y^n}{n!}
\Bigr)^{-1}=\sum_{n\ge 0} \sum_{\sigma \in \sym_{n}}e_{\lambda(\sigma)}(\xx)\dfrac{y^n}{n!}$ 
(\cite{Dleon2016})};
\draw[->] (a) -- (b)node [midway, above, sloped] (T1) {$p_k(\xx)=\delta_{1,k}$};

\node (c) at (3.5,2) {$\Bigl 
(1-h_1(\yy)+\sum_{n\ge 2}(-1)^n t(t-1)^{n-1}h_{n}(\yy)
\Bigr)^{-1}=\sum_{n\ge 0} \sum_{\alpha \in \comp_{n}}s_{H(\alpha)}(\yy)t^{\ell(\alpha)}$ (Equation 
\ref*{equation:specializatione2only})};
\draw[->] (a) -- (c)node [midway, above, sloped] (T2) {$e_k(\xx)=t$};

\node (e) at (7,0) {$\dfrac{1-t}{1-te^{(1-t)y}}= \sum_{n\ge0}A_n(t)\dfrac{y^n}{n!}$ 
(\cite{Riordan1951})};
\draw[->] (b) -- (e)node [midway, above, sloped] (T3) {$e_k(\xx)=t$};

\draw[->] (c) -- (e)node [midway, above, sloped] (T4) {$p_k(\xx)=\delta_{1,k}$};

\node (d) at (3.5,-1) {$(1-h_1(\yy))^{-1}=\sum_{n\ge 0} \sum_{\alpha \in 
\comp_{n}}s_{H(\alpha)}(\yy)=\sum_{n\ge 0}\ch(\CC[\sym_n])$ (Equation 
\ref*{equation:specializatione2andt1})};
\draw[->] (c) -- (d)node [midway, above, sloped] (T1) {$t=1$};
\end{tikzpicture}
\caption{Diagram of specializations}\label{figure:diagramspecializations}
\end{figure}
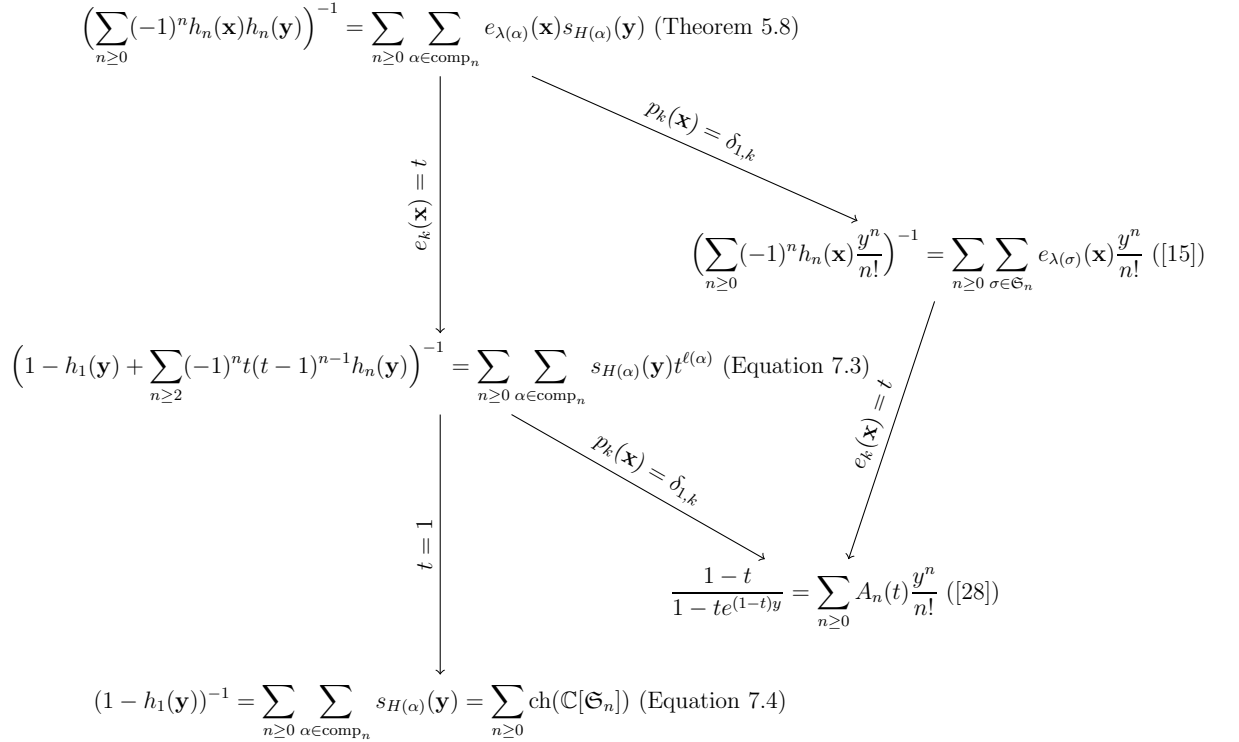

\section*{Acknowledgments}
The author is grateful to Ira Gessel for the very valuable discussions and explanations 
related to the technique for inversion of power series involved in the proof of 
Theorem \ref{theorem:explicitmultiplicativeinverse}. The author would like to thank as well an 
anonymous reviewer for their helpful comments and suggestions. Most of the work on this project was 
done while the author was a postdoctoral scholar at University of Kentucky and the author is 
grateful for its support.

\bibliographystyle{abbrv}
\bibliography{altalgebra}

\end{document}